\documentclass[a4paper,12pt]{amsart}
\pdfoutput=1
\usepackage[utf8]{inputenc}
\usepackage[T1]{fontenc}
\usepackage{amssymb,mathtools}
\usepackage{mathrsfs}
\usepackage{dsfont}
\usepackage{tensor} 
\usepackage{lmodern}
\usepackage{color}
\usepackage{comment}
\usepackage[normalem]{ulem}
\usepackage{tikz}
\usepackage{geometry}
\usepackage[all]{xy}
\usepackage{microtype}

\usepackage{enumitem}
\setlist[enumerate,1]{label=\textup{(\arabic*)}}

\usepackage[lite]{amsrefs}
\newcommand*{\MRref}[2]{ \href{http://www.ams.org/mathscinet-getitem?mr=#1}{MR \textbf{#1}}}

\renewcommand*{\PrintDOI}[1]{\href{http://dx.doi.org/\detokenize{#1}}{doi: \detokenize{#1}}}

\BibSpec{book}{%
  +{}  {\PrintPrimary}                {transition}
  +{,} { \textit}                     {title}
  +{.} { }                            {part}
  +{:} { \textit}                     {subtitle}
  +{,} { \PrintEdition}               {edition}
  +{}  { \PrintEditorsB}              {editor}
  +{,} { \PrintTranslatorsC}          {translator}
  +{,} { \PrintContributions}         {contribution}
  +{,} { }                            {series}
  +{,} { \voltext}                    {volume}
  +{,} { }                            {publisher}
  +{,} { }                            {organization}
  +{,} { }                            {address}
  +{,} { \PrintDateB}                 {date}
  +{,} { }                            {status}
  +{}  { \parenthesize}               {language}
  +{}  { \PrintTranslation}           {translation}
  +{;} { \PrintReprint}               {reprint}
  +{.} { }                            {note}
  +{.} {}                             {transition}
  +{} { \PrintDOI}                   {doi}
  +{} { available at \url}            {eprint}
  +{}  {\SentenceSpace \PrintReviews} {review}
}

\usepackage[pdftitle={Toeplitz algebras of product systems},
pdfauthor={},
pdfsubject={Mathematics}
]{hyperref}

\numberwithin{equation}{section}

\theoremstyle{plain}
\newtheorem{thm}[equation]{Theorem}
\newtheorem{cor}[equation]{Corollary}
\newtheorem{lem}[equation]{Lemma}
\newtheorem{prop}[equation]{Proposition}

\theoremstyle{definition}
\newtheorem{defn}[equation]{Definition}
\newtheorem{note}[equation]{Notation}

\theoremstyle{remark}
\newtheorem{rem}[equation]{Remark}
\newtheorem{example}[equation]{Example}

\newcommand{\ZZ}{\mathbb{Z}}

\newcommand{\RR}{\mathbb{R}}
\newcommand{\QQ}{\mathbb{Q}}
\newcommand{\NN}{\mathbb{N}}

\newcommand{\mx}{\mathcal{X}}
\def\X{\mx}
\newcommand{\my}{\mathcal{Y}}
\def\Y{\my}
\newcommand{\me}{\mathcal{E}}
\def\E{\me}

\newcommand*{\nb}{\nobreakdash}
\newcommand*{\Star}{\(^*\)\nobreakdash-}
\newcommand{\Cst}{\mathrm{C}^*}
\newcommand{\idealin}{\mathrel{\triangleleft}} 
\newcommand*{\Bound}{\mathbb B}
\newcommand*{\Comp}{\mathbb K}

\newcommand{\Toepr}{\mathcal{T}_\lambda}
\newcommand{\Toepu}{\mathcal{T}_u}
\newcommand{\id}{\mathrm{id}}

\newcommand*\xbar[1]{%
   \hbox{%
     \vbox{%
       \hrule height 0.5pt 
       \kern0.5ex
       \hbox{%
         \kern-0.1em
         \ensuremath{#1}%
         \kern-0.1em
       }%
     }%
   }%
} 

\DeclarePairedDelimiterX{\braket}[2]{\langle}{\rangle}{#1\,\delimsize\vert\,\mathopen{}#2}
\DeclarePairedDelimiterX{\BRAKET}[2]{\langle}{\rangle}{\!\delimsize\langle#1\,\delimsize\vert\,\mathopen{}#2\delimsize\rangle\!}

\DeclarePairedDelimiterX{\setgiven}[2]{\{}{\}}{#1\,{:}\,\mathopen{}#2}

\newcommand{\thmref}[1]{Theorem~\ref{#1}}
\newcommand{\secref}[1]{Section~\ref{#1}}
\newcommand{\proref}[1]{Proposition~\ref{#1}}
\newcommand{\lemref}[1]{Lemma~\ref{#1}}
\newcommand{\corref}[1]{Corollary~\ref{#1}}
\newcommand{\remref}[1]{Remark~\ref{#1}}
\newcommand{\defref}[1]{Definition~\ref{#1}}

\def\W{\mathcal W}
\def\inv{^{-1}}

\def\iqs{Q}
\def\J{\mathfrak J}

\def\mc{\mathcal C}

\def\fock{\phi}

\def\mi{\mathcal I}
\def\tmi{\tilde\mi}
\def\tma{\tilde\ma}
\def\lsp{\operatorname{span}}
\def\clsp{\operatorname{\overline {span}}}
\def\Fl{\Cst_{\mathrm{rep}}(\E)}
\def\ma{\mathfrak A}

\newcommand{\matr}[3]{\begin{pmatrix}
#1 & #2 \\ 0 & #3
\end{pmatrix}
}
\begin{document}
\title[On Toeplitz algebras of product systems]{On Toeplitz algebras of product systems}

\author{Elias G. Katsoulis}

\address{Department of Mathematics, East Carolina University, Greenville NC 27858-4353 USA}
\thanks{}

\email{KATSOULISE@ecu.edu}
\author{Marcelo Laca}

\address{Department of Mathematics and Statistics, Univ. of Victoria, BC  Canada V8W 2Y2}
\thanks{}
\email{laca@uvic.ca}

\author{Camila F. Sehnem}

\address{Department of Pure Mathematics, Univ. of Waterloo, Waterloo ON Canada N2L 3G1}
\thanks{}

\email{camila.sehnem@uwaterloo.ca}
\date{\today}
\dedicatory{To the memory of Iain Raeburn}
\begin{abstract} 
In the setting of product systems over group-embeddable monoids, we consider nuclearity of the associated Toeplitz C*-algebra in relation to nuclearity of the coefficient algebra. Our work goes beyond the known cases of single correspondences and compactly aligned product systems over right LCM monoids. Specifically, given a product system over a submonoid of a group, we show, under technical assumptions, that the fixed-point algebra of the gauge action is nuclear iff the coefficient algebra is nuclear; when the group is amenable, we conclude that this happens iff the Toeplitz algebra itself is nuclear.  Our main results imply that nuclearity of the Toeplitz algebra is equivalent to nuclearity of the coefficient algebra for every full product system of Hilbert bimodules over abelian monoids,  over $ax+b$-monoids of integral domains and over Baumslag--Solitar monoids $BS^+(m,n)$ that admit an amenable embedding,  which we provide for $m$ and $n$ relatively prime.

\end{abstract}
\maketitle

\section{Introduction}
C*-algebras associated to correspondences have been actively studied ever since they were introduced by Pimsner \cite{Pimsner:Generalizing_Cuntz-Krieger}. Originally promoted as simultaneous generalizations of Cuntz--Krieger algebras and crossed products by $\ZZ$, Pimsner algebras have found many other applications and stimulated a lot of research over the past three decades. The construction uses the tensor powers of  a correspondence over a C*-algebra of coefficients to produce a C*-algebra of left creation operators (the Toeplitz--Pimsner algebra) and a distinguished quotient of it, (the Cuntz--Pimsner algebra). Pimsner's approach also relies on viewing  these algebras as universal objects for a class of representations of the correspondence. For the details, in addition to the original reference, see  also \cites{Fow-Rae1999,Katsura2004}. The construction appeared independently in \cite{AEE} under the further assumption that the correspondence is a Hilbert bimodule. 

 Partly guided by Arveson's theory of product systems of Hilbert spaces, Fowler \cite{Fowler:Product_systems} 
 proposed a generalization of Pimsner's construction for generalized systems of tensor products of correspondences over semigroups ({\em product systems}, for short). The definition of a reduced Toeplitz algebra generated by left creation operators on the  Fock module of a product system is entirely analogous to the case of a single correspondence, for which the relevant semigroup is $\NN$. But
complications arose when trying to identify a convenient class of representations of the product system that was useful to study the reduced Toeplitz algebra from a universal C*-algebra point of view. This was the case even for semigroup C*-algebras, which arise from product systems of one-dimensional correspondences; essentially, the reason is that some crucial properties of representations are automatic for $\NN$ but need to be specified explicitly in general to keep things tractable.

With the theory of semigroup C*-algebras available at the time, the construction of a universal Toeplitz algebra associated to a product system could only proceed under further assumptions, broadly, that the semigroup arises from a quasi-lattice order in the sense of Nica and that the product system is compactly aligned. It was only under these assumptions, or slight generalizations thereof, that a workable theory was developed in analogy to Pimsner's.

The past 15 years have seen significant advances in the understanding of semigroup C*-algebras beyond the quasi-lattice or right LCM conditions. This has been largely the result of seminal work of Xin Li~\cite{Li:Semigroup_amenability}, which takes into account the structure of constructible right ideals of the semigroup. Naturally, these advances are now beginning to permeate the study of C*-algebras of general product systems, notably through the techniques developed in \cites{SEHNEM2019558}. This provides the context for the present work, in which  we study product systems of correspondences over monoids that embed in a group.
  
  In this paper we are primarily interested in the structure of the (reduced) Toeplitz algebra $\Toepr(\E)$ of a product system $\E$ over a general group-embeddable monoid~$P$. Our main goal is to study nuclearity and exactness of $\Toepr(\E)$ in terms of the underlying coefficient algebra beyond the case of compactly aligned product systems over right LCM monoids. For compactly aligned product systems over right LCM monoids, an analysis of these structural properties is carried out in~\cite{KKLL2023}. For single correspondences, nuclearity (resp. exactness) of the Toeplitz algebra was shown to be equivalent to nuclearity (resp. exactness) of the coefficient algebra in work of Katsura~\cite{Katsura2004} (see also \cite{DS2001}). For similar structural results at the level of Cuntz--Pimsner-type C*-algebras of product systems over Ore monoids we mention work of Albandik and Meyer \cite{albandik-meyer}, and pioneering work of Murphy on C*-algebras associated to abelian semigroups of endomorphisms~\cite{Murphy1996IEOT}.

  This paper is organized as follows. We begin \secref{sec:covariance4productsystems} with a very brief introduction to Toeplitz C*-algebras of product systems. We then give the definition of Fock covariant representations and of a universal Toeplitz algebra, and discuss a filtration of C*-algebras indexed by the collection of finite families of constructible ideals that are closed under intersection. In \secref{sec:tensorproducts} we consider tensor products of product systems. Our motivation is a result in the setting of Toeplitz algebras of correspondences that establishes the existence of canonical isomorhisms $\Toepr(\E)\otimes_{\mu}B\cong\Toepr(\E\otimes_{\mu} B)$ where $\mu$ denotes either the minimal or maximal norms.  We show that an isomorphism $\Toepr(\E)\otimes_{\min}B\cong\Toepr(\E\otimes_{\min} B)$ always exists, and discuss what happens with the maximal norm in relation to universal Toeplitz algebras.

  The independence condition for monoids introduced by Li~\cite{Li:Semigroup_amenability} is relevant in our work and we review it in \secref{sec:independenceandps}. We discuss related properties in the setting of product systems that allow us to give a simplified characterization of Fock covariant representations. When $P$ embeds in a group $G$ we analyze the fixed point algebra $\Toepr(\E)^G$ for the canonical gauge coaction of~$G$ in a similar fashion as for Toeplitz algebras of single correspondences in work of Katsura \cite{Katsura2004} and of compactly aligned product systems in work of Fowler~\cite{Fowler:Product_systems}. We give sufficient conditions for nuclearity of the fixed point algebra $\Toepr(\E)^G$ for the gauge coaction in terms of subalgebras associated to constructible right ideals of~$P$. In many cases these conditions are also necessary, \thmref{thm:building-subalgebras}. In \thmref{thm:building-subalgebras-exactness} we give sufficient conditions for exactness of $\Toepr(\E)^G$.

\secref{sec:main-results} contains our main results. We show that if $\E=(\E_p)_{p\in P}$ is a product system with coefficient algebra~$A$ for which the left action on $\E_p$ contains the compact operators for each $p\in P$, then $\Toepr(\E)^G$ is nuclear if and only if $A$ is nuclear provided that~$\Toepr(\E)^G$ is spanned in a certain precise sense by elements associated to symmetric words in~$P$, see~\thmref{thm:Toep-vs-coefficient1}. The main idea is that these conditions imply that each of the building blocks of $\Toepr(\E)^G$ corresponding to a constructible right ideal of $P$ is a quotient of an ideal in~$A$. The analogous statement about exactness is proved in \thmref{thm:Toep-vs-coefficient-exactness}. We observe that the left action on~$\E_p$ contains the compact operators for all $p\in P$ whenever $\E=(\E_p)_{p\in P}$ is a product system of Hilbert bimodules. 

In \secref{sec:examples} we give several examples of monoids for which every full product system such that the left action contains the compact operators automatically satisfies the conditions in our main theorem. These include abelian monoids, $ax+b$\nb-monoids of integral domains and Baumslag--Solitar monoids $BS^+(m,n)$.
In the case  $\gcd(m,n) =1$ we give an embedding of $BS^+(m,n)$ into an amenable group; it is possible that this is known but we have found no reference for it. In any case, it is of  independent interest.

\subsection*{Acknowledgements} Part of this work was carried out in person around the International Workshop in Operator Theory and its Applications IWOTA 2022; we are thankful for the hospitality extended to us by Marek Ptak at Uniwersytet Rolinczy. Elias Katsoulis was partially supported by NSF Grant 2054781 and Marcelo Laca was supported by NSERC Discovery Grant RGPIN-2023-05410. At the later stages we received the preprint \cite{KaPa2024} which also deals with Fock covariance for product systems; the overlap is limited because the main concern there is with Hao--Ng type theorems, while the main results here are about nuclearity and exactness.

\section{Product systems and Fock covariance}\label{sec:covariance4productsystems}

\subsection{Notation and basic notions} Let $A$ and $B$ be $\Cst$\nb-algebras. A \emph{correspondence} $\E\colon A\leadsto B$ is a Hilbert $B$\nb-module~$\E$ with a nondegenerate left action of~$A$ implemented by a \Star homomorphism $\fock\colon A\to\Bound(\E)$. We say that $\E$ is a \emph{Hilbert $A,B$\nb-bimodule} if the left action of~$A$ comes from a left Hilbert $A$\nb-module structure $\BRAKET{\cdot}{\cdot}_A$ on~$\E$ that is compatible with $\braket{\cdot}{\cdot}_B$. That is, $\BRAKET{\xi}{\eta}\zeta=\xi\braket{\eta}{\zeta}$ for all $\xi,\eta,\zeta\in\E$. We say that $\E$ is \emph{full} if $\braket{\E}{\E}=A$.

 Let~$P$ be a semigroup with identity~$e$. A \emph{product system} over~$P$ of $A$\nb-correspondences consists of:
\begin{enumerate}
\item[(i)] a correspondence $\E_p\colon A\leadsto A$ for each~$p\in P$, where $\E_e=A$  is the identity correspondence over~$A$;
\item[(ii)] correspondence isomorphisms $\mu_{p,q}\colon\E_p\otimes_A\E_q\overset{\cong}{\rightarrow}\E_{pq}$, also called \emph{multiplication maps}, for all $p,q\in P$; $\mu_{e,p}$ on $A\otimes_A\E_p$ will be left multiplication $\varphi_p\colon A\to\Bound(\E_p)$, while $\mu_{p,e}$ will be the right action of~$A$ on~$\E_p$
\end{enumerate}

 The multiplication maps are associative, that is, the following diagram commutes for all $p, q, r\in P$:
  \[
  \xymatrix{
    (\E_p\otimes_A\E_q)\otimes_A\E_r  \ar@{->}[d]^{\mu_{p,q}\otimes1}
   \ar@{<->}[rr]& &     \E_p\otimes_A(\E_q\otimes_A\E_r)
   \ar@{->}[rr]^{1\otimes\mu_{q,r}}&&
    \E_p\otimes_A\E_{qr} \ar@{->}[d]^{\mu_{p,qr}} \\
    \E_{pq}\otimes_A\E_r   \ar@{->}[rrrr]^{\mu_{pq,r}}&& &&
    \E_{pqr}  
.  }
  \]
  
We say that $\E$ is \emph{full} if $\E_p$ is full for all $p\in P$. If each~$\E_p$ is a Hilbert $A$\nb-bimodule, we speak of a \emph{product system of Hilbert bimodules}.

 A \emph{representation} of a product system~$\E=(\E_p)_{p\in P}$ in a $\Cst$\nb-algebra~$B$ consists of linear maps~$\psi_p\colon\E_p\rightarrow B$, for all~$p\in P\setminus\{e\},$ and a \Star homomorphism $\psi_e\colon A\rightarrow B$, satisfying 
  \begin{enumerate}
  \item[(R1)] $\psi_p(\xi)\psi_q(\eta)=\psi_{pq}(\xi\eta)$ for all $p,q\in P$, $\xi\in\E_p$ and  $\eta\in\E_q$;
  \item[(R2)] $\psi_p(\xi)^*\psi_p(\eta)=\psi_e(\braket{\xi}{\eta})$ for all $p\in P$ and $\xi, \eta\in\E_p$.  
  \end{enumerate}
    
 If $\psi_e$ is faithful, we say that~$\psi$ is \emph{injective}. In this case, the relation (R2) implies that~$\|\psi_p(\xi)\|=\|\xi\|$ for all~$\xi\in\E_p$ and~$p\in P$.

\bigskip
The {\em reduced Toeplitz algebra of $\E$} is defined as follows. Let $\E^+ = \bigoplus_p \E_p$ be the Fock correspondence, and let $\fock$
   be the Fock representation of $\E$ on $\E^+$ given by  creation operators $\fock_p(\xi)$ for $\xi \in \E_p$ defined 
  on fibers by  $\fock_p(\xi) \eta = \xi\otimes \eta$. In particular $\fock_e$ implements left multiplication by~$A$. These operators satisfy (R1) and (R2) hence give a representation of~$\E$, see \cite[p. 340]{Fowler:Product_systems}.

 \bigskip
 We will  need to consider the universal C*-algebra $\Fl$ of representations of $\E$ satisfying (R1) and (R2). This   was introduced by Fowler 
 in \cite{Fowler:Product_systems}, who referred to it as the Toeplitz algebra of $\E$. We believe that a Toeplitz algebra ought to reflect further
 properties of the reduced Toeplitz C*-algebra $\Toepr(\E)$. Thus, we will refer to $\Fl$ simply as the universal C*-algebra of representations of 
 $\E$, and reserve the term Toeplitz for a C*-algebra whose representations satisfy a further set of relations derived from the Fock representation.
 To illustrate the issue we point out that it is generally accepted that the Toeplitz C*-algebra of the monoid $\NN^2$ is generated by two commuting isometries $S$ and $T$ that also satisfy $S^*T = TS^*$, and that omitting this last condition produces a universal C*-algebra that is much larger than the intended target \cite{MR1386163}.

\bigskip  

  \subsection{Fock covariance}  
  As in many  situations involving concrete C*-algebras of operators, it is very useful here to  have an abstract presentation of a universal object in terms of generators and relations that single out a class of representations. The concrete C*-algebra is then a canonical image of the universal one, and faithfulness depends on a type of amenability.
  In the case of 
   compactly aligned product systems over quasi-lattice ordered semigroups this was achieved in \cite{Fowler:Product_systems} by considering  Nica--Toepitz covariant representations. Our next aim is to come up with such a presentation  suitable for $\Toepr(\E)$ when $\E$ is a product system over a group-embeddable monoid. The idea is to add more relations to the presentation (R1)-(R2) of $\Fl$.
 
 Recall from \cite[Proposition~4.1]{DKKLL2022} that if $P$ embeds in a group~$G$, then the reduced Toeplitz algebra $\Toepr(\E)$ has a canonical gauge coaction of $G$ determined by 
 \[
 \fock_p(\xi) \mapsto \fock_p(\xi) \otimes u_p, \qquad p \in P,  \  \  \xi\in \E_p
 \]
  where $u_p$ denotes the canonical generator of $C^*(G)$ associated to $p$. 
 
    \bigskip 
   
   \begin{defn}\label{def:CMcovariance}(Fock covariance)
   Let $\psi^u$ denote the universal representation of $\E$ in $\Fl$ and recall that the map $\psi^u(\xi_p) \mapsto \psi^u(\xi_p) \otimes \delta_p$ determines a natural gauge coaction of $G$ on the universal C*-algebra  $\Fl$ for representations of $\E$. 
   We will say that a representation $\psi$ of $\E$ is {\em Fock covariant}  if the restriction of the associated representation of $\Fl$ 
 to the fixed point algebra $\Fl^G$  of the gauge coaction factors through the fixed point algebra $\Toepr(\E)^G$.

   Following \cite{laca-sehnem} we define the universal Toeplitz algebra of $\E$,  denoted by $\Toepu(\E)$, to be the universal C*-algebra for Fock covariant representations. We will write $\phi^u :\E \to \Toepu(\E)$ for the universal Fock covariant representation.  Obviously the left regular representation $\fock$ itself is Fock covariant, and this gives a canonical homomorphism of $\Toepu(\E)$ onto $\Toepr(\E)$ that restricts to an isomorphism $\Toepu(\E)^G\cong\Toepr(\E)^G$.
   \end{defn}
   
   We will  give a  C*-algebraic characterization of Fock covariance in \proref{pro:equiv-subalgebras} below.  Before we can state it, we need to 
   consider certain subalgebras of $\Fl$ and of $\Toepu(\E)$ defined using constructible right ideals of $P$, in analogy to what is done for semigroup C*\nb-algebras \cite{Li:Semigroup_amenability}. It will be convenient to use the notation from \cite[Section~2]{laca-sehnem}, which we review next. For each word $\alpha = (p_1, p_2, \cdots , p_{2k-1}, p_{2k}) \in W^k$ we let $\tilde{\alpha} = (p_{2k},  p_{2k-1},  \cdots ,p_2, p_1)$ be the reverse word, and define $\dot\alpha \coloneqq p_1\inv p_2 ,  \cdots, p_{2k-1}\inv p_{2k}$;  we say that the word $\alpha$ is {\em neutral} if $\dot\alpha =e$. The iterated quotient set is the finite set
   \[
\iqs(\alpha) \coloneqq \{ e, \  p_{2k}\inv  p_{2k-1} , \,   p_{2k}\inv  p_{2k-1} p_{2k-2}\inv  p_{2k-3}, \, \ldots, \, p_{2k}\inv  p_{2k-1} p_{2k-2}\inv  p_{2k-3} \cdots p_2\inv p_1  \},
\]
and  the corresponding constructible right ideal  is
   \[
   K(\alpha) \coloneqq \bigcap_{g\in \iqs(\alpha)} gP = P \, \cap   p_{2k}\inv  p_{2k-1} P \,\cap  \, p_{2k}\inv  p_{2k-1} p_{2k-2}\inv  p_{2k-3} P \,\cap  \ldots \cap \,\dot{\tilde{\alpha}} P.
   \]
   
Denote by $\psi^u$ the universal representation of $\E$ in $\Fl$. For each neutral word $\alpha = (p_1, p_2, \cdots , p_{2k-1}, p_{2k})$ we consider products in  $\Fl$  of the form
  \[
\psi^u(\xi_{p_1})^* \psi^u(\xi_{p_2})  \cdots \psi^u(\xi_{p_{2k-1}})^* \psi^u(\xi_{p_{2k}}) 
 \]
 in which $\xi_{p_j} \in \E_{p_j}$ for each $j = 1, 2, \ldots, 2k$.
 

\begin{lem}\label{lem:single-ideals-sub} Let $P$ be a submonoid of a group $G$ and let $\E = (\E_p)_{p\in P}$ be a product system of correspondences over the $\Cst$\nb-algebra $A$.
Then for each $S\in \J$ the space 
\[
\tmi(S)\coloneqq \clsp\{\psi^u(\xi_{p_1})^* \psi^u(\xi_{p_2})  \cdots \psi^u(\xi_{p_{2k-1}})^* \psi^u(\xi_{p_{2k}}) \mid\dot\alpha = e \text{ and } K(\alpha) = S\}
\]\
 is  a C*-subalgebra of $\Fl^G$ and 
\begin{enumerate}
 \item $\phi(\tmi(\emptyset) )= (0)$ if $\emptyset\in\J$;

\smallskip\item $\tmi(R) \tmi(S) \subset \tmi(R\cap S)$ for $R,S \in \J$;
\smallskip\item $\tmi(P) =A$.
\end{enumerate}
\begin{proof}
By definition $\tmi(S)$ is a closed  subspace of $\Fl$. To see that it is  $*$-closed, 
notice that the adjoint of the spanning element $\psi^u(\xi_{p_1})^* \psi^u(\xi_{p_2})  \cdots \psi^u(\xi_{p_{2k-1}})^* \psi^u(\xi_{p_{2k}})$
corresponding to a word $\alpha = (p_1, p_2, \cdots , p_{2k-1}, p_{2k})$ satisfying $\dot\alpha =e$ and $K(\alpha) = S$
is precisely another spanning element corresponding to the word $\tilde\alpha$, which satisfies $\dot{\tilde\alpha} = \dot{\alpha}\inv = e$ and $K(\tilde\alpha)=K(\alpha)$  by e.g. \cite[Lemma 2.5(4)]{laca-sehnem}.

Suppose now $R$ and $S$ are constructible ideals in $P$ and let $\alpha =(p_i)$ and $\beta = (q_j)$ be neutral words satisfying $K(\alpha) = R$
and $K(\beta) = S$. A typical spanning element of $\tmi(R) \tmi(S)$ is a product of the form
\[
\psi^u(\xi_{p_1})^* \psi^u(\xi_{p_2})  \cdots \psi^u(\xi_{p_{2k-1}})^* \psi^u(\xi_{p_{2k}}) \psi^u(\xi_{q_1})^* \psi^u(\xi_{q_2})  \cdots \psi^u(\xi_{q_{2l-1}})^* \psi^u(\xi_{q_{2l}}),
\]
 which corresponds to the concatenation $\alpha\beta =(p_1, \ldots, q_{2l}) \in \W^{k+l}$. Since $(\alpha\beta)\dot{} = \dot\alpha \dot\beta =e$ and $K(\alpha \beta) = K(\alpha) \cap K(\beta) = R\cap S$, 
  by e.g. \cite[Lemma 2.5(4)]{laca-sehnem},  we conclude that (2) holds. Setting $R = S$ now shows 
the spanning set, and hence $\tmi(S)$ is closed under multiplication.

For part (3), let $\alpha=(p_i)$ be a neutral word with  $K(\alpha)=P$, and let 
$\xi_{p_j}\in\E_{p_j}$ for $j=1, \ldots 2k$. 
We will show by induction on $k$ that there exists $a\in A$ such that  
\begin{equation}\label{eqn:oldxidotalpha}
  \psi^u(\xi_{p_1})^* \psi^u(\xi_{p_2})  \cdots \psi^u(\xi_{p_{2k-1}})^* \psi^u(\xi_{p_{2k}})=a.  
\end{equation}

Indeed, for $k=1$ we must have $\alpha=(p,p)$ and the conclusion follows because $\psi^u$ satisfies axiom (R2) of the definition of representation of $\E$. 
Suppose the claim holds for words of length up to $k-1$ and let  $\alpha = (p_1, p_2, \cdots , p_{2k-1}, p_{2k})$.
Since $e\in K(\alpha)\subset p_{2k}^{-1}p_{2k-1} P\cap P$,  there exists $r\in P$ such that $p_{2k}=p_{2k-1}r$, and hence $\psi^u(\xi_{p_{2k-1}})^*\psi^u(\xi_{p_{2k}})= \psi^u (\eta)$ for some $\eta\in \E_r$. Using this and (R1) we can condense  the last three factors in \eqref{eqn:oldxidotalpha}, obtaining
\[
\psi^u(\xi_{p_{1}})^* \cdots \psi^u(\xi_{p_{2k-2}})   \psi^u(\xi_{p_{2k-1}})^* \psi^u(\xi_{p_{2k}})= \psi^u(\xi_{p_{1}})^* \cdots \psi^u(\xi_{p_{2k-2}}\eta) 
\]
where the right hand side has length $k-1$ and is thus in $A$  by the induction hypothesis. This completes the proof of part (3) and of the proposition.
\end{proof}
\end{lem}

\begin{cor} The projection of the Fock space onto the direct summand $\E_e=A$ induces conditional expectations $E_{A,u}\colon \Toepu(\E)\to A$ and $E_{A,r}\colon \Toepr(\E)\to A$ that factor through the corresponding conditional expectations onto $\Toepu(\E)^G\cong\Toepr(\E)^G$, and vanish on $\tmi(S)$ for all $S\in \J\setminus \{P\}.$
\begin{proof} This follows from \lemref{lem:single-ideals-sub}(3).
    \end{proof}
    \end{cor}

\begin{lem}\label{lem:direct-limit} \label{pro:Cstar-sum} Let $P$ be a submonoid of a group $G$ and let $\E$ be a product system over~$P$. For each finite $\cap$-closed  subcollection $\mc$ of constructible right ideals define 
\[
\tma(\mc) \coloneqq \lsp \{b_S\in\tmi(S)\mid S\in\mc\}.
\] 
 Then 
 \begin{enumerate}
     \item   $\tma(\mc)$ is a C*-algebra (i.e. the linear span is already closed); 

         \item  if $\mc_1 \subset \mc_2$, then $\tma(\mc_1)$ is an ideal in $ \tma( \mc_2)$;
         
     \item  $\{\tma(\mc)\}_\mc$ is an increasing system of C*-subalgebras indexed by the set of  finite $\cap$-closed subcollections $\mc$ of $\J$ partially ordered by inclusion, whose union is dense in $\Fl^G$; 
 
     \item 
for each  maximal element $\bar{S}$ of $\mc$, the sequence \[
0\to \tma(\mc\setminus\{\bar{S}\} ) \to\tma(\mc)\to\tmi(\bar{S})/(\tmi(\bar{S})\cap\tma(\mc\setminus\{\bar{S}\}))\to 0
\]
is exact.
\end{enumerate}
\end{lem}

\begin{proof} It is clear that $\tma(\mc_1) \subset \tma(\mc_2) $ whenever $\mc_1 \subset \mc_2$, and that $\tma(\mc) \subset \Fl^G $.  
Since $\Fl^G$ is generated by products corresponding to neutral words, $\bigcup_\mc \tma(\mc)$ is dense in $\Fl^G$.

Let $\mc$ be a finite $\cap$-closed  collection of constructible right ideals of $P$. We will  show that $\tma(\mc) 
$ is a $\Cst$\nb-subalgebra of~$\Cst_{\mathrm{rep}}(\E)$. The proof is by induction  on~$|\mc|$. If $|\mc|=1$, then $\tma(\mc)=\tmi({S})$, which  is a $\Cst$\nb-algebra by \lemref{lem:single-ideals-sub}. Let $n\geq 1$ and suppose $\sum_{\substack{S\in \mc'}}\tmi(S)$ is a $\Cst$\nb-algebra, whenever $\mc'\subset\J$ is a finite $\cap$-closed subcollection of $\J$ with $|\mc'|=n$.  Assume $|\mc|=n+1$ and let $\bar{S}$ be a maximal element of $\mc$. Set $\mc'\coloneqq \mc\setminus\{\bar{S}\}$. Then $|\mc'|=n$ and $\mc'$ is still closed under intersection because by removing the maximal element $\bar{S}$ we have not removed any intersection of elements of $\mc'$. By the induction hypothesis, $\tma(\mc')$ is a $\Cst$\nb-algebra. Moreover, $\tma(\mc')$ is a closed ideal in~$\tma(\mc)$ because  $\tmi(\bar{S})\tmi(S)\subseteq \tmi(\bar{S}\cap S)$ for every $S$ in $\mc'$  by Lemma~\ref{lem:single-ideals-sub}, and $\bar{S}\cap S$ has to be in $\mc'$ . Thus $\sum_{\substack{S\in\mc}}\tmi(S)=\tmi(\bar{S})+\tma(\mc')$ is a $\Cst$\nb-algebra by e.g. \cite[Theorem~3.1.7]{murphy1990c}. Since $\tma(\mc')$ is a closed ideal in~$\tma(\mc)$ there is an  isomorphism $\tma(\mc)/\tma(\mc')\cong \tmi(\bar{S})/(\tmi(\bar{S})\cap\tma(\mc'))$ which proves that the given sequence is exact.
\end{proof}

We also define the analogous subalgebras at the level of $\Toepr(\E)^G \cong \Toepu(\E)^G$, namely
\[
\mi(S)\coloneqq \clsp\{\fock(\xi_{p_1})^* \fock(\xi_{p_2})  \cdots \fock(\xi_{p_{2k-1}})^* \fock(\xi_{p_{2k}}) \mid \dot\alpha = e \text{ and } K(\alpha) = S\}
\] 
and
\[
\ma(\mc) \coloneqq \clsp \{\mi(S) \mid S\in \mc\}.
\]
These are the images in $\Toepr(\E)^G$ of  $\tmi(S)$ and $\tma(\mc)$ from \lemref{lem:single-ideals-sub} and \lemref{lem:direct-limit};  and
it is clear that the analogues of both lemmas also hold at the level of $ \Toepr(\E)^G$.

\begin{rem}\label{rem:sameFPAforallG}
Along the lines indicated in \cite[Remark~3.11]{SEHNEM2019558} for the fixed point algebra of the covariance algebra of a product system, here too, the fixed point algebra $\Toepr(\E)^G$ of the reduced Toeplitz algebra does not depend on the specific group $G$. This is important because it implies that the notion of Fock covariance from \defref{def:CMcovariance} is intrinsic to the product system itself. 
In order to see why, recall from the proof of \cite[Proposition~4.1]{DKKLL2022}  that $\Toepr(\E)^G$ is the closed linear span of the set of elements  of the form 
\[
\fock(\xi_{p_1})^* \fock(\xi_{p_2})  \cdots \fock(\xi_{p_{2k-1}})^* \fock(\xi_{p_{2k}}) \qquad \text{ with } p_1\inv p_2 \cdots p_{2k-1}\inv p_{2k} = e. 
\]
By \cite[Lemma~2.9]{laca-sehnem} this set does not depend on the specific $G$ into which $P$ embeds because when $K(\alpha) = \emptyset$ for a neutral word in $P$ the product vanishes by \lemref{lem:single-ideals-sub}(1). 
\end{rem}

\subsection{Coaction on $\Toepu(\E)$} In order to conclude that the coaction of $G$ on $\Fl$ passes to the universal Toeplitz C*-algebra we verify next that the ideal of $\Fl$ generated by the kernel of the restriction of the Fock representation to $\Fl^G$ is gauge invariant.

\begin{lem}[cf.  \cite{SEHNEM2019558}*{Lemma~3.3}] Let $J_e^{\fock}\coloneqq \ker\tilde{\fock}\cap\Cst_{\mathrm{rep}}(\E)^G$ and let $J^{\fock}\idealin\Cst_{\mathrm{rep}}(\E)$ be the ideal generated by~$J_e^{\fock}$. Then 
\[
J^{\fock}=\bigoplus_{\substack{g\in G}}\Cst_{\mathrm{rep}}(\E)^gJ_e^{\fock}=\bigoplus_{\substack{g\in G}}J^{\fock}\cap \Cst_{\mathrm{rep}}(\E)^g,
\]
where $\Cst_{\mathrm{rep}}(\E)^g$ denotes the spectral subspace at~$g$ relative to the canonical coaction of~$G$.
\end{lem}
\begin{proof} We begin by proving the first equality. To do so, it suffices to show that $J_e^{\fock} \Cst_{\mathrm{rep}}(\E)^g\subseteq  \Cst_{\mathrm{rep}}(\E)^g J_e^{\fock}$. Take $c_g\in  \Cst_{\mathrm{rep}}(\E)^g $ and $b\in J_e^{\fock}$. Since $\tilde{\fock}$ is a \Star homomorphism, $bc_g$ belongs to the kernel of the Fock representation. So $c_g^*b^*bc_g\in \ker\tilde{\fock}\cap\Cst_{\mathrm{rep}}(\E)^G$. And so $bc_g\in \Cst_{\mathrm{rep}}(\E)^gJ_e^{\fock}$ by \cite[Lemma~3.5]{Pimsner:Generalizing_Cuntz-Krieger} (view $\Cst_{\mathrm{rep}}(\E)^g$ as a correspondence over~$\Cst_{\mathrm{rep}}(\E)^G$). Using the first equality and the contractive projection from $\Cst_{\mathrm{rep}}(\E)$ onto $\Cst_{\mathrm{rep}}(\E)^g$, we deduce that $J^{\fock}\cap \Cst_{\mathrm{rep}}(\E)^g=\Cst_{\mathrm{rep}}(\E)^gJ_e^{\fock}$ and hence the second equality also holds.
\end{proof}

As a consequence of the previous lemma, we have the following:

\begin{lem}[cf. \cite{SEHNEM2019558}*{Lemma~3.4}] Let $q_{\mathrm{cov}}\colon\Cst_{\mathrm{rep}}(\E)/ J^{\fock} \to \Toepu(\E)$ be the quotient map. There is a full nondegenerate coaction $\delta\colon\Cst_{\mathrm{rep}}(\E)/  J^{\fock}\rightarrow\Cst_{\mathrm{rep}}(\E)/ J^{\fock}\otimes\Cst(G)$ satisfying $\delta\circ q_{\mathrm{cov}}=(q_{\mathrm{cov}}\otimes\id_{\Cst(G)})\circ\widetilde{\delta}$, where $\widetilde{\delta}$ stands for the coaction of~$G$ on~$\Cst_{\mathrm{rep}}(\E)$. Moreover, the spectral subspace for~$\delta$ at~$g\in G$  is canonically isomorphic to~$\Cst_{\mathrm{rep}}(\E)^g/\Cst_{\mathrm{rep}}(\E)^gJ_e^{\fock}.$ 
\begin{proof} The proof follows as in \cite[Lemma~3.4]{SEHNEM2019558}.
\end{proof}
\end{lem}

\begin{prop}\label{pro:equiv-subalgebras}
Let $\psi\colon\E \to B$ be a representation of the product system $\E$ in the C*-algebra $B$, and denote by $\tilde\psi$ the corresponding 
homomorphism of $\Fl$ onto the C*-algebra generated by $\psi(\E)$. Then $\psi$ 
is Fock covariant if and only if for every $\mc$ the restriction $\tilde\psi: \tma(\mc) \to B$ factors through $\ma(\mc) := \tilde\fock(\tma(\mc))$ (image of $\tma(\mc)$ in the Fock representation). 
\begin{proof} The ``only if'' direction is clear because $\tma(\mc) \subset \Fl^G$. We will prove that a representation $\psi$ of $\E$ is Fock covariant whenever its restriction $\tilde{\psi}\restriction_{\tma(\mc)}$ factors through the Fock representation for every  finite $\cap$-closed $\mc\subset\J$. It suffices to show that $\tilde{\psi}$ vanishes on  $J_e^{\fock}\coloneqq \ker\tilde{\fock}\cap\Cst_{\mathrm{rep}}(\E)^G$. By Lemma~\ref{lem:direct-limit}, $\Cst_{\mathrm{rep}}(\E)^G=\lim_{\mc} \tma(\mc)$. Thus $J_e^{\fock}=\lim_{\mc}J_e^{\fock}\cap \tma(\mc)$ by e.g. \cite[Lemma~1.3]{Adji-Laca-Nilsen-Raeburn}. Since $\tilde{\psi}$ vanishes on $J_e^{\fock}\cap \tma(\mc)$ for every~$\mc$ by assumption, we conclude that $\psi$ is covariant.
\end{proof}
\end{prop}

\section{Tensor products}  \label{sec:tensorproducts}
Suppose that $\E = (\E_p)_{p\in P}$ is a product system of correspondences over a C*-algebra $A$ and that $B$ is another C*-algebra, and choose a norm $\mu$ on the algebraic tensor product $A \odot B$. Then the system of algebraic tensor products $\E \odot B = \{ \E_p\odot B\}_{p\in P}$ can be completed to a product system $\E\otimes_{\mu} B=(\E_p\otimes_\mu B)_{p\in P}$ over $A\otimes_\mu B$ with the inner product and the norm on  $\E_p \odot B$ determined by the norm $\mu$ on $A\otimes_\mu B$ via 
\[
\langle\xi\otimes b, \eta\otimes c\rangle = \langle\xi, \eta\rangle \otimes \langle b, c\rangle
\qquad \text{ for }\xi,\eta \in \E_p \text{ and  } b,c\in B.\] 

Motivated by the case of single correspondences and also of C*-algebras associated to Fell bundles, in this section we investigate the relationships between $\Toepu(\E)\otimes_{\max}B$ and $\Toepu(\E\otimes_{\max} B)$, and similarly between $\Toepr(\E)\otimes_{\min} B$ and $\Toepr(\E\otimes_{\min} B)$. For a single correspondence, there are canonical isomorphisms $\Toepr(\E)\otimes_{\max}B\cong\Toepr(\E\otimes_{\max} B)$ and $\Toepr(\E)\otimes_{\min}B\cong\Toepr(\E\otimes_{\min} B)$. We prove below that the canonical isomorphism at the level of minimal tensor products and reduced Toeplitz algebras always exists in the setting of product systems, \proref{pro:tensorB}. However, complications arise from both the more general structure of the semigroup and of the product system. So at the level of maximal tensor products and universal Toeplitz algebras we are only able to establish the existence of a canonical homomorphism from $\Toepu(\E)\otimes_{\max}B$ onto $\Toepu(\E\otimes_{\max} B)$, \proref{pro:surjective}. 

\bigskip
 
 Initially we will consider the more general situation of the external tensor product of two product systems, regarded as a product system over the direct product of monoids with coefficients in the minimal tensor product of the coefficient algebras. 
 
  Suppose that $A$ and $B$ are C*-algebras, $\X$ is a Hilbert $A$-module and $\Y$ is a Hilbert $B$-module. Recall that the exterior tensor product $\X\otimes_{\min} \Y$ is the completion of the algebraic tensor product of $\X$ and $\Y$, with respect to the norm arising from the $A\otimes_{\min} B$-valued sesquilinear form defined on elementary tensors by
\[
 \langle\xi_1\otimes \eta_1, \xi_2\otimes \eta_2\rangle = \langle\xi_1, \xi_2\rangle \otimes \langle \eta_1, \eta_2\rangle \in A\otimes_{\min} B
\]
The fact that the above sesquilinear form is actually an inner product is proven in \cite[pg. 34]{Lan}; there it is also shown that  $\X\otimes_{\min} \Y$ admits a (natural) right module action by $A\otimes_{\min} B$ and so it becomes a Hilbert $A\otimes_{\min} B$-module . Furthermore, there is a $*$-injection 
\[
j \colon \Bound(\X)\otimes_{\min} \Bound(\Y)\longrightarrow \Bound(\X\otimes_{\min} \Y); s\otimes t \longmapsto j(s\otimes t),
\]
with $j(s\otimes t) (\xi \otimes \eta)= s\xi \otimes t \eta$, $\xi \in \X.\eta \in \Y$. (For a proof of this fact see \cite[pg. 35-37]{Lan}.) The existence of the injection $j$ easily implies that if $\X$ and $\Y$ are C*-correspondences over $A$ and $B$ respectively, then $\X\otimes_{\min} \Y$ becomes an $A\otimes_{\min} B$-correspondence.

\begin{prop}\label{pro:tensorY}
 Let $\X$ and $\Y$ be  product systems over monoids $P$ and $Q$ with coefficient C*-algebras $A$ and  $B$, respectively, then there is a product system $\X \otimes_{\min} \Y$ over the monoid $P\times Q$  with coefficient C*-algebra $A\otimes_{\min} B$ in which the fiber over $(p,q)$ is 
 $ \X_p\otimes \Y_q $ and the multiplication 
 is given by 
 \[
(\xi_1\otimes \eta_1)_{(p_1,q_1)} (\xi_2\otimes \eta_2)_{(p_2, q_2)} = (\xi_1 \xi_2)_{p_1p_2} \otimes  (\eta_1 \eta_2)_{q_1q_2}.
     \] 
This gives  a  canonical isomorphism of Fock modules,
\[
 (\X \otimes_{\min} \Y)^+ \cong \X^+ \otimes_{\min} \Y^+
 \]
 and a spatial isomorphism of reduced Toeplitz algebras,
 \[
  \Toepr(\X \otimes_{\min} \Y) \cong \Toepr(\X) \otimes_{\min} \Toepr(\Y).
 \]
 
 \begin{proof}
 It follows from the previous discussion that each fiber $(\X\otimes_{\min}\Y)_{p,q}$, $p,q \in P$, is an $A\otimes_{\min} B$-correspondence. The balanced tensor product
\[
(\X\otimes_{\min}  \Y)_{p_1,q_1}\underset{A\otimes_{\min}  B}{\otimes}(\X\otimes_{\min}  \Y)_{p_2,q_2}
\]
with respect to $A\otimes B$ corresponds to the tensor product of balanced tensors 
\[
(\X_{p_1}\underset{A}{\otimes} \X_{p_2}) \otimes (\Y_{q_1}\underset{B}{\otimes} \Y_{q_2}).
\]
 Hence if we define a multiplication on $\X\otimes_{\min} \Y$  by
     \[
(\xi_1\otimes \eta_1)_{(p_1,q_1)} (\xi_2\otimes \eta_2)_{(p_2, q_2)} = (\xi_1 \xi_2) \otimes  (\eta_1 \eta_2) \in \X_{p_1 p_2} \otimes \Y_{q_1q_2},
     \]
with $\xi_i \in \X_{p_i}$, $\eta_i \in \X_{q_i}$ $i=1,2$, then this multiplication is associative and therefore $\X\otimes_{\min} \Y$ becomes a product system.

It is easy to see now that there is an isometric isomorphism of right-Hilbert $A\otimes B$ Fock bimodules
\[
 (\X \otimes_{\min} \Y)^+ \cong \X^+ \otimes_{\min} \Y^+.
 \]
 Since for all elementary tensors $\xi_i \otimes \eta_i\in \X_{p_i} \otimes_{\min} \Y_{q_i}$ for $i =1,2$,
\begin{align*}
    L_{\xi_1 \otimes \eta_1} (\xi_2 \otimes \eta_2)  &= (\xi_1 \otimes \eta_1) \underset{A\otimes B}{\otimes} (\xi_2 \otimes \eta_2)
    = 
 (\xi_1  \underset{A}{\otimes} \xi_2) \otimes (\eta_1 \underset{ B}{\otimes} \eta_2))\\
 &=
 L_{\xi_1} \xi_2 \otimes  L_{\eta_1}  \eta_2
 = (L_{\xi_1} \otimes  L_{\eta_1} ) ( \xi_2 \otimes   \eta_2)
\end{align*}
 the map 
\[
L_{\xi_1 \otimes \eta_1} \mapsto 
L_{\xi_1} \otimes_{\min} L_{\eta_1}
 \]
 extends to the required isomorphism of reduced Toeplitz algebras.
 \end{proof}
 \end{prop}

 \begin{prop} \label{pro:tensorB}
 Let $\E$ be a product system over $A$ and let $B$ be a C*-algebra, then there are canonical isomorphisms $\E^+ \otimes B \cong (\E\otimes B)^+$ and 
 \[
 \Toepr(\E) \otimes_{\min} B \cong \Toepr(\E \otimes_{\min} B).
 \]
 \begin{proof}
 The result follows from \proref{pro:tensorY} with  $\X =\E$, $\Y =B$ and $Q =\{e\}$ so that $P\times Q \cong P$ and $\Toepr(\Y) =B$.
 \end{proof}
 \end{prop}
 
 \begin{prop}\label{pro:surjective} 
Suppose $\E = (\E_p)_{p\in P}$ is a product system over the monoid $P$ with coefficients in $A$, and let $B$ be a C*-algebra. Then there is  a (canonical) homomorphism 
\[
\pi : \Toepu(\E)\otimes_{\max}B \rightarrow  \Toepu(\E\otimes_{\max}B)  . 
\]
such that $\pi: T_\xi \otimes b \mapsto T_{\xi\otimes b} $ for $\xi\in \E_p$ and $b \in B$.
\begin{proof}
The map that sends $\xi\in \E_p$ to  $T_{\xi\otimes 1} \in \Toepu(\E\otimes_{\max}B)$ for each $p\in P$ is a Fock covariant representation of $\E$ that commutes with the representation $b\mapsto 1 \otimes b$ of $B$ in the multiplier algebra  $\mathcal M (\Toepu(\E\otimes_{\max}B))$.
The universal property of the maximal tensor product then gives a homomorphism $\pi: \Toepu(\E)\otimes_{\max}B
\to \mathcal M(\Toepu(\E\otimes_{\max}B))$ whose image is generated by the products $T_{\xi\otimes 1} (1 \otimes b) = T_{\xi\otimes b}$ and hence is contained in $\Toepu(\E\otimes_{\max}B)$. 
\end{proof}

 \end{prop}

 \begin{prop}\label{pro:faithful-action} Let $\E = (\E_p)_{p\in P}$ be a product system over a  C*-algebra $A$  and let $B$ be a C*-algebra. The following are equivalent:
\begin{enumerate}
\smallskip \item the homomorphism $\pi \colon \Toepu(\E)\otimes_{\max}B \rightarrow  \Toepu(\E\otimes_{\max}B)$ is an isomorphism;
 
\smallskip \item the composition of $\pi$ with the Fock representation of $\Toepu(\E\otimes_{\max}B)$ is faithful on  $\Toepu(\E)^G \otimes_{\max}B$.

\end{enumerate}
\begin{proof}
    (1) $\implies$(2) is obvious because the Fock representation of $\Toepu(\E\otimes_{\max}B)$ is faithful on $\Toepu(\E\otimes_{\max}B)^G$. Suppose (2) holds. Then $\pi$ restricts to an isomorphism $\Toepu(\E)^G\otimes_{\max}B \cong \Toepu(\E\otimes_{\max}B)^G$. It follows that the representation of $\E\otimes B$ in $\Toepu(\E)^G\otimes_{\max}B $ that identifies elementary tensors is Fock covariant, and hence it induces a homomorphism $\Toepu(\E\otimes_{\max}B)\to \Toepu(\E)^G\otimes_{\max}B $ by universal property. This is the inverse of $\pi$.
\end{proof}
\end{prop}

\begin{cor} \label{cor:isomorphism-string}
Let $\E = (\E_p)_{p\in P}$ be a product system over the monoid $P$ with coefficient C*-algebra $A$.  Suppose that $A$ is nuclear. Then the following are equivalent:

\begin{enumerate}
\smallskip \item the homomorphism $\pi \colon \Toepu(\E)\otimes_{\max}B \rightarrow  \Toepu(\E\otimes B)$ is an isomorphism for every C*-algebra $B$;
 
\smallskip \item $\Toepu(\E)^G$ is nuclear.

\end{enumerate}
If, in addition, $\Toepu(\E\otimes B)\cong \Toepr(\E\otimes B)$ canonically for every C*-algebra~$B$ (e.g. if $G$ is amenable), the above conditions are equivalent to:
\begin{enumerate}
    \item[\textup{(3)}] $\Toepr(\E)$ is nuclear.
\end{enumerate}
\begin{proof} If $\pi$ is an isomorphism for every C*-algebra~$B$, restriction to the fixed point algebra $\Toepu(\E)^G\otimes_{\max}B$ implies that $\Toepu(\E)^G\otimes_{\max}B\cong \Toepu(E\otimes B)^G$ canonically for every C*-algebra~$B$. Then the assumption combined with \proref{pro:tensorB} yields $$\Toepu(\E)^G\otimes_{\max}B\cong \Toepu(E\otimes B)^G\cong\Toepr(\E\otimes B)^G\cong \Toepu(\E)^G\otimes_{\min}B$$ canonically for every C*-algebra~$B$, proving that $\Toepu(\E)^G$ is nuclear. 

Conversely, suppose $\Toepu(\E)^G$ is nuclear. Then using again \proref{pro:tensorB} we have for every C*-algebra $B$ that $$\Toepu(\E)^G\otimes_{\max}B\cong\Toepu(\E)^G\otimes_{\min}B\cong \Toepr(\E)^G\otimes_{\min}B\cong \Toepr(\E\otimes B)^G\cong \Toepu(\E\otimes B)^G$$ canonically. \proref{pro:faithful-action} gives that $\pi$ is an isomorphism.
\end{proof}
\end{cor}

\begin{rem}  Notice that when $G$ is amenable, the equivalence between conditions $(2)$ and~$(3)$ in \corref{cor:isomorphism-string} follows from a general result on C*-algebras associated to Fell bundles, see \cite[Proposition~25.10]{Exel:Partial_dynamical}. Under the assumption that $A$ is nuclear, Condition $(1)$ in \corref{cor:isomorphism-string} always holds if $\E$ is a single correspondence or, more generally, a compactly aligned product system over a right LCM monoid. We do not know if \corref{cor:isomorphism-string}$(1)$ automatically holds for arbitrary product systems (over group embeddable monoids).
    
\end{rem}

\section{Independence and the fixed point algebra} \label{sec:independenceandps}
 Recall from \cite[Definition~2.26]{Li:Semigroup_amenability} that  a  monoid~$P$ is said to satisfy independence 
0 when no constructible right ideal of $P$ can be written as a finite union of proper sub-ideals. Equivalently, $P$ satisfies independence if and only if  the  projections $\{\chi_S\mid S\in\J\}$ in $\mathcal{B}(\ell^2(P))$ corresponding to characteristic functions on constructible right ideals of~$P$ form a linearly independent set. In this case the Fock covariance condition for the canonical product system over~$P$ is equivalent to a condition involving single right ideals in~$\J$, rather than subcollections of them. Motivated by this observation, we consider product systems $\E$ such that 
\begin{equation}\label{eqn:independenceps1}
\ma(\mc)=\bigoplus_{\substack{S\in\mc}}\mi(S)
\qquad \text{ for every finite $\cap$-\nb closed }\mc\subset\J.
\end{equation}

\begin{prop} \label{pro:covariance-independence} Let $P$ be a submonoid of a group and let $\E$ be a product system over~$P$.   If a representation $\psi$ of~$\E$ in a $\Cst$\nb-algebra~$B$ is Fock covariant, then  the restriction of~$\tilde{\psi}$ to~$\tmi(S)$ factors through the Fock representation $\tilde{\fock}$ for every~$S\in\J$.
The converse holds if $\E$ satisfies \eqref{eqn:independenceps1}.
\end{prop}
\begin{proof} The first assertion follows easily from \proref{pro:equiv-subalgebras} with $\mc = \{S\}$. For the converse, assume \eqref{eqn:independenceps1} holds and let  $\psi$ be a representation of~$\E$ in~$B$ such that $\tilde{\psi}\restriction_{\tmi(S)}$ factors through the Fock representation for every~$S\in\J$. To show that~$\psi$ is Fock covariant, by \proref{pro:equiv-subalgebras} all we need to prove is that $\tilde{\psi}\restriction_{\tma(\mc)}$ factors through the Fock representation for every finite $\cap$\nb-closed subcollection $\mc\subseteq\J$. Specifically, we need to show that  $\tilde{\psi}(b)=0$ for each $b\in \tma(\mc)$ such that $\tilde{\fock}(b)=0$. By \lemref{lem:direct-limit}, we can write $b=\sum_{\substack{S\in\mc}}b_S$ with $b_S\in\tmi(S)$ for each~$S\in\mc$. It follows that 
$$0=\tilde{\fock}(b)=\sum_{\substack{S\in\mc}}\tilde{\fock}(b_S).$$ 
But then $\tilde{\fock}(b_S)=0$ for each $S\in\mc$ by \eqref{eqn:independenceps1}. By assumption, $\tilde{\psi}\restriction_{\tmi(S)}$ factors through $\tilde{\fock}$ and thus $\tilde{\psi}(b_S)=0$ which proves that $\tilde{\psi}(b)=0$,  as desired.
\end{proof}

\subsection{Sufficient conditions and examples}
Next we observe that the property \eqref{eqn:independenceps1} for product systems is consistent with independence for monoids.

\begin{prop}\label{pro:consistentindependence} \cite[see Proposition~5.6.22]{CELY} Let $P$ be a submonoid of a group $G$. Then the canonical product system over~$P$ with one-dimensional fibres satisfies \eqref{eqn:independenceps1} if and only if~$P$ satisfies independence.

\begin{prop}\label{pro:suff-conditions} Let $P$ be a submonoid of a group $G$ and assume that $P$ satisfies independence. Suppose further that the family $\J$ of constructible right ideals of~$P$ has the following property: given $S_1,S_2\in \J$ with $S_1\subsetneq S_2$ and $s_1\in S_1$, there exists $s_2\in S_2\setminus S_1$ such that $s_2\leq s_1$. Then every product system over~$P$ satisfies \eqref{eqn:independenceps1}.
\end{prop}
\begin{proof} First, we claim that if $S$, $S_i$, $i=1,\ldots, n$ are in~$\J$, and $S\setminus(\bigcup_{\substack{j}}S_j)\neq\emptyset$, then for every $s\in \bigcup^n_{\substack{j=1}}S\cap S_j$, there exists $\bar{s}\in S\setminus(\bigcup^n_{\substack{j=1}}S_j)$ with $\bar{s}\leq s$. We will prove the claim by induction on~$n$. This is just the hypothesis on the family~$\J$ of ideals of~$P$ if $n=1$. Assume the claim holds for~$n\geq 1$ fixed. Let $S$, $S_j$, $j=1,\ldots,n+1$ be constructible right ideals of~$P$ and suppose that $$S\setminus\big(\bigcup^{n+1}_{\substack{j=1}}S_j\big)\neq\emptyset.$$ Take $s\in \bigcup^{n+1}_{\substack{j=1}}S\cap S_j $ and let $i\in\{1,\ldots,n+1\}$ be such that $s\in S\cap S_i$. Let $s'\in S\setminus(S\cap S_i)$ with $s'\leq s$. If $s'\in S\setminus(\bigcup^{n+1}_{\substack{j=1}}S_j)$ we are done. Otherwise, $s'\in \bigcup_{\substack{i\neq j}}S\cap S_j$ and we can apply the induction hypothesis to find $\bar{s}\in S\setminus\bigcup_{\substack{j\neq i}} S_j$ with $\bar{s}\leq s'\leq s$.  Because $s'\not\in S_i$ and $S_i$ is a right ideal, $\bar{s}$ cannot lie in~$S_i$ as well. Thus $\bar{s}\leq s$ with $\bar{s}$ in $S\setminus(\bigcup^{n+1}_{\substack{j=1}}S_j)$. This completes the proof of the claim.

Now let~$b=\sum_{\substack{S\in\mc}}b_S\in \tma(\mc)$ with $\tilde{\fock}(b)=0$ in $\mathcal{B}(\E^+)$. We claim that $\tilde{\fock}(b_S)=0$ for every~$S\in\mc$. Let $\bar{S}$ be a maximal element of $\mc$ and choose  $\bar s \in \bar{S} \setminus \bigcup_{\substack{\bar{S}\neq S\in\mc}}S$. It follows that, for all $S\in\mc\setminus\{\bar{S}\}$, $\tilde{\fock}(b_S)=0$ on the direct summand $\E_{\bar{s}}\subseteq \E^+$ because $\bar{s}\not\in S$. Thus $\tilde{\fock}(b_{\bar{S}})$ must vanish on~$\E_{\bar{s}}$ whenever $\bar{s}\not\in\bigcup_{\substack{\bar{S}\neq S\in\mc}}S$. By the first part of the proof, if $s\in \bar{S}\cap \bigcup_{\substack{\bar{S}\neq S\in\mc}}S$, we can find $\bar{s}\in  \bar{S}\setminus\big(\bar{S}\cap \bigcup_{\substack{\bar{S}\neq S\in\mc}}S\big)$ with $\bar{s}\leq s$. Using the correspondence isomorphism $\E_{s}\cong\E_{\bar{s}}\otimes\E_{\bar{s}^{-1}s}$, we deduce that $\tilde{\fock}(b_{\bar{S}})=0$ on~$\E_s$. So $\tilde{\fock}(b_{\bar{S}})=0$. Proceeding with this argument, we conclude that $\tilde{\fock}(b_S)=0$ for each $S\in\mc$. So $\E$ satisfies independence as asserted.
\end{proof}
\end{prop}
 
\begin{cor} Let $P\subseteq G$ and suppose that $P$ is right LCM. Then every product system over $P$ satisfies \eqref{eqn:independenceps1}. 
\begin{proof} Let $P$ be a right LCM submonoid of a group $G$. Then every right ideal of~$P$ is principal and $P$ satisfies independence by \cite[Lemma~5.6.31]{CELY}. Let $S_1$ and $S_2$ 
 be nonempty right ideals of~$P$ such that $S_1\subsetneq S_2$. We can write $S_1=pP$ and $S_2=qP$ with $p,q\in P$ such that $q < p$.  For every $s_1 \in S_1$ we may choose $s_2 \coloneqq q \in S_2 \setminus S_1$, which clearly satisfies $s_2 < p \leq s_1$. 
The result follows by Proposition~\ref{pro:suff-conditions}.
\end{proof}
\end{cor}

\subsection{Exactness and nuclearity in terms of subalgebras} 
Suppose that $P$ is a submonoid of a group $G$ and let $S$ be a constructible right ideal in $P$. The family 
\[
\J_{\subset S}\coloneqq\{R\in \J\mid R\subset S\}
\] of all constructible right ideals contained in $S$ is $\cap$-closed but in most cases is infinite. Since $\J_{\subset S}$ is relevant to the study of  nuclearity and exactness of $\Toepr(\E)^G$,  we introduce in \thmref{thm:building-subalgebras} a stronger version of condition \eqref{eqn:independenceps1} that takes it into account. Let $\E$ be a product system over $P$  and define
\[
\ma(\J_{\subset S})\coloneqq \clsp\{b\in \Toepr(\E)^G\mid b\in \mi(R), R\in \J_{\subset S}\}.
\] 
Even when $\J_{\subset S}$ is infinite we have 
\[
\ma(\J_{\subset S})= \lim_{\mc\Subset \J_{\subset S}} \ma(\mc) .
\] 
\begin{lem}\label{lem:Cstar-sum-subset} Let $P$ be a submonoid of a group $G$, and let $\E$ be a product system over~$P$. 
Then $\ma(\J_{\subset S})$ and $\ma(\J_{\subset S}\setminus \{S\})$ are ideals in $\Toepr(\E)^G$, and the sequence
\begin{equation}\label{eqn:exactCS}
0\to \ma(\J_{\subset S}\setminus\{S\})\to\ma(\J_{\subset S})\to\mi(S)/(\mi(S)\cap\ma(\J_{\subset S}\setminus\{S\}))\to 0
\end{equation}
is exact.
\end{lem}
\begin{proof} Since the constructible right ideal $S$ is  maximal (in fact largest)  in $\J_{\subset S}$, the statement follows as in part (4) of \lemref{lem:direct-limit}.
\end{proof}

\begin{thm}\label{thm:building-subalgebras} Let $P$ be a submonoid of a group $G$ and let $\E$ be a product system over~$P$. Suppose that for each constructible right ideal $S\in \J$ the $\Cst$\nb-algebra $\mi(S)$ is nuclear. Then the fixed point algebra $\Toepr(\E)^G$ of the gauge coaction is nuclear. The converse holds if  
\begin{equation}\label{eqn:moreindependent}
\mi(S)\cap\ma(\J_{\subset S}\setminus\{S\}) = (0) \qquad \text{for all } S\in \J.
\end{equation}

\begin{proof} Suppose first that $\mi(S)$ is nuclear for every $S\in \J.$ An induction argument on $|\mc|$ combined with \proref{pro:Cstar-sum} shows that $\ma(\mc)$ is nuclear for every  finite $\cap$\nb-closed collection $\mc$ of constructible right ideals of $P$. Since $\Toepr(\E)^G=\lim_{\mc}\ma(\mc)$, it follows that $\Toepr(\E)^G$ is nuclear.

Assume now that \eqref{eqn:moreindependent} holds, so that the  exact sequence \eqref{eqn:exactCS} reads
\[0\to \ma(\J_{\subset S}\setminus\{S\})\to\ma(\J_{\subset S})\to\mi(S)\to 0.
\]
If $\Toepr(\E)^G$ is nuclear, then $\ma(\J_{\subset S})$ is nuclear because it is an ideal in $\Toepr(\E)^G$;
but then  $\mi(S)$ is nuclear because it is a quotient of $\ma(\J_{\subset S})$ by \lemref{lem:Cstar-sum-subset}.
    \end{proof}
    
\end{thm}

Since exactness passes to subalgebras, it is clear that $\mi(S)$ is exact for every $S\in \J$ provided that $\Toepu(E)^G$ is exact. To prove the converse we need \eqref{eqn:independenceps1}.

\begin{thm}\label{thm:building-subalgebras-exactness} Let $P$ be a submonoid of a group $G$ and let $\E$ be a product system over~$P$. Suppose that $\E$ satisfies \eqref{eqn:independenceps1} and that the $\Cst$\nb-algebra $\mi(S)$ is exact for every constructible right ideal $S\in \J$. Then $\Toepr(\E)^G$ is exact.
 
 \begin{proof} Suppose that $\mi(S)$ is exact for every $S\in \J$ and let $\mc\subset \J$ be a finite $\cap$\nb-closed collection of constructible right ideals. By \eqref{eqn:independenceps1} we have $\mi(\bar{S})\cap \ma(\mc\setminus\{\bar{S}\})=(0)$ when $\bar{S}$ is a maximal element in $\mc$. Hence the analogue of the exact sequence of \lemref{pro:Cstar-sum}(4) at the level of $\Toepr(\E)^G$ splits. Since exactness is preserved by  taking split extensions, an induction argument on $|\mc|$ shows that $\ma(\mc)$ is exact for every  finite $\cap$\nb-closed collection $\mc$ of constructible right ideals. Since exactness is preserved by taking direct limits, it follows that $\Toepr(\E)^G=\lim_{\mc}\ma(\mc)$ is exact.
    \end{proof}
    
\end{thm}

    Condition \eqref{eqn:moreindependent} implies condition \eqref{eqn:independenceps1}, and both are satisfied by right LCM monoids. If we apply \thmref{thm:building-subalgebras} to a compactly aligned product system over a right LCM monoid we recover a key result that is embedded in the proof of \cite[Theorem~6.11]{KKLL2023}.
\begin{cor}[cf. \cite{KKLL2023}*{Theorem~6.11}]\label{cor:compactly-align} Let $P$ be a right LCM submonoid of $G$ and let $\E$ be a compactly aligned product system over $P$ with coefficient algebra~$A$. Then the fixed point algebra $\Toepr(\E)^G$ of the gauge coaction is nuclear (resp. exact) if and only if $A$ is nuclear (resp. exact).
\begin{proof} If $\E$ is compactly aligned, for every $p\in P$ the C*-subalgebra $\mi(pP)$ corresponding to~$S=pP$ is the algebra of compact operators~$\Comp(\E_p)$. This is nuclear (resp. exact) if and only if $A$ is,  because $\Comp(\E_p)$ is Morita equivalent to an ideal in~$A$. The conclusion now follows from \thmref{thm:building-subalgebras} and \thmref{thm:building-subalgebras-exactness}.
    
\end{proof}

\end{cor}

\section{Main results}\label{sec:main-results}

We wish to apply the  results in the previous section to a class of product systems for which the left action of the coefficient algebra contains the compact operators. Central to our approach is the set $\W_{\rm{sym}}$ of 
symmetric words. Recall that a word $\beta \in \W$ is symmetric if it is the form $\beta =\tilde\alpha \alpha$ for $\alpha \in \W$, where $\tilde{\alpha}$ is the reverse of $\alpha$.

\bigskip

\begin{note} We will denote by $\langle\W_{\rm{sym}}\rangle$ the smallest subcollection of words in $\W$ containing all symmetric words that is closed under concatenation and under conjugation by words in $\W$, in the sense that if $\alpha\in \W$ and $\beta\in \langle\W_{\rm{sym}}\rangle $, then $\tilde{\alpha}\beta\alpha\in \langle\W_{\rm{sym}}\rangle.$ Given a constructible right ideal $S\in \J$, we denote by $\langle\W_{\rm{sym}}\rangle (S)$ the collection of words $\beta\in \langle\W_{\rm{sym}}\rangle$ such that $K(\beta)=S.$ 
Notice that every word in $\langle\W_{\rm{sym}}\rangle$ is neutral.
Further, given a word $\beta=(p_1,p_2,\ldots, p_{2l-1},p_{2l})\in\W$, we let $$\mi(\beta)\coloneqq\clsp\{\fock(\xi_{p_1})^*\fock(\xi_{p_2})\ldots \fock(\xi_{p_{2l-1}})^* \fock(\xi_{p_{2l}})\mid \xi_{p_i}\in \E_{p_1}\text{ for }i=1,\ldots,2l\}.$$ 
\end{note}

\begin{lem}\label{lem:end-symmetric-words1} Let $P$ be a submonoid of a group and let $\E=(\E_p)_{p\in P}$ be a product system over~$P$ with coefficients in  a $\Cst$\nb-algebra $A$. Suppose  that $\varphi_p(A)\supset \Comp(\E_p)$ for all $p\in P$ and let $\beta\in \langle\W_{\rm{sym}}\rangle$. Then for each $b\in \mi(\beta)$ there exists $a\in A$ such that $b = \varphi_s(a)$ as operators on $\E_s$ for every $s\in K(\beta)$. 
\begin{proof} 
The result is true if $\beta=(p,p)$ for $p\in P$ because then $\fock(\xi_p)^*\fock(\eta_p)=\fock(\braket{\xi_p}{\eta_p})$ for every $\xi_p, \eta_p\in \E_p$. Suppose $\beta\in \W$ has the form $\beta=(p,q,q,p).$ Let $\xi_1,\xi_2\in \E_p$ and $\eta_1,\eta_2\in \E_q.$ We will show that there exists $a \in A$ such that 
\[
\fock(\xi_1)^* \fock(\eta_1) \fock(\eta_2)^*\fock(\xi_2)\zeta_s=\varphi_s(a)\zeta_s
\]
for all $\zeta_s\in \E_s$ and $s\in K(\beta)$. 

Let $c\in A$ be such that $\varphi_s(c)=\fock(\eta_1) \fock(\eta_2)^*$ on $\E_s$ for all $s\in qP$. If $s\in K(\beta)$, then $ps\in qP$ and we have 
\[
\fock(\xi_1)^* \fock(\eta_1) \fock(\eta_2)^*\fock(\xi_2)\zeta_s=\fock(\xi_1)^* \varphi_{ps}(c)\fock(\xi_2)\zeta_s=\varphi_s(\braket{\xi_1}{\varphi_p(c)\xi_2})\zeta_s.
\]
Setting $a\coloneqq \braket{\xi_1}{\varphi_p(c)\xi_2}$ we see that $\fock(\xi_1)^* \fock(\eta_1) \fock(\eta_2)^*\fock(\eta_2)\zeta_s=\varphi_s(a)\zeta_s$ for all $\zeta_s\in \E_s$ and all $s\in K(\beta)$, as needed.

In order to establish the statement, consider the collection of neutral words $\beta$ such that for all $b\in \mi(\beta)$ there exists $a\in A$ with $\varphi_s(a)=b$ on $\E_s$ for all $s\in K(\beta)$. To show that this collection contains $\langle\W_{\rm{sym}}\rangle$, it suffices to show that it is closed under concatenation and conjugation  by words of the form $(p,q)$ with $p,q\in P$. Indeed,  that it is closed under concatenation follows because $K(\beta_1\beta_2)=K(\beta_1)\cap K(\beta_2)$ whenever $\beta_1,\beta_2\in \W$ are neutral words and because the diagonal action of $A$ on the Fock space is multiplicative. That the collection is closed under concatenation by words of the form $(p,q)$ with $p,q\in P$ follows as in the first part of the proof.
\end{proof}

\end{lem}

\begin{prop}\label{pro:nuclearity-exactness-sym1} Let $P$ be a submonoid of a group and let $\E=(\E_p)_{p\in P}$ be a product system over~$P$ with coefficients in  a $\Cst$\nb-algebra $A$. Suppose that $\varphi_p(A)\supset \Comp(\E_p)$ for all $p\in P$. Let $\emptyset\neq S\in \J$ and suppose that 
\begin{equation}\label{eqn:I(S)includedinspan}
    \mi(S)\subset\clsp\{bQ_S\mid b\in \mi(\beta),\  \beta\in \langle\W_{\rm{sym}}\rangle, \ K(\beta)\supset S\},
    \end{equation}
where $Q_S$ denotes the orthogonal projection of the Fock space~$\E^+$ onto $\bigoplus_{p\in S}\E_p$. Then there exists an ideal $I_S\idealin A$ such that the \Star homomorphism $\fock_S$ from $A$ into the algebra of adjointable operators on $\bigoplus_{s\in S}\E_s$ obtained by the restriction of the Fock representation induces a surjective \Star homomorphism  $\fock_S\restriction_{I_S}\colon I_S\to \mi(S)$. In particular, if $A$ is nuclear (resp. exact) then so is $\mi(S)$. 
\begin{proof}
     Let $\fock_S$ be the \Star homomorphism from $A$ into the algebra of adjointable operators on $\bigoplus_{s\in S}\E_s$ obtained by compressing the Fock representation with the projection onto $\bigoplus_{s\in S}\E_s$. By \lemref{lem:end-symmetric-words1} for every $\beta\in\langle\W_{\rm{sym}}\rangle$ and $b\in \mi(\beta)$ there exists $a\in A$ such that $\fock_{K(\beta)}(a)=b$. Then the assumption \eqref{eqn:I(S)includedinspan} implies that $\fock_S(A)\supset \mi(S)$.
     
     Define $I_S\coloneqq \{a\in A\mid \fock_S(a)\in \mi(S)\}.$ Then $I_S$ is an ideal of~$A$ because $\fock^+$ is a representation of $\E$. We conclude that $\fock_S\restriction_{I_S}\colon I_S\to \mi(S)$  is a surjective \Star homomorphism. Since ideals of nuclear $\Cst$\nb-algebras are also nuclear and nuclearity passes to quotients, we obtain that nuclearity of $A$ implies that of $\mi(S).$ 
\end{proof}
\end{prop}

    \begin{thm}\label{thm:Toep-vs-coefficient1}  Let $P$ be a submonoid of a group $G$ and let $\E=(\E_p)_{p\in P}$ be a product system over~$P$ with coefficients in  a $\Cst$\nb-algebra $A$. Suppose that $\varphi_p(A)\supset \Comp(\E_p)$ for all $p\in P$ and that  every nontrivial constructible right ideal $ S$ of $P$ satisfies
    \eqref{eqn:I(S)includedinspan}. 
    Then the following are equivalent:
    \begin{enumerate}
        \item the coefficient algebra $A$ is nuclear;
        \item the fixed point algebra $\Toepr(\E)^G$ is nuclear.
    \end{enumerate} 
    If, in addition, $G$ is amenable, $\Toepu(\E) \cong \Toepr(\E)$ and these conditions are also equivalent to
    \begin{enumerate}
        \item[\textup{(3)}] the Toeplitz algebra $\Toepr(\E)$ is nuclear.
    \end{enumerate} 
\begin{proof} 
Suppose first $A$ is nuclear. By \proref{pro:nuclearity-exactness-sym1},  $\mi(S)$ is nuclear for every $S\in \J$. By \thmref{thm:building-subalgebras} $\Toepr(\E)^G$ is nuclear. 

Suppose now that $\Toepr(\E)^G$ is nuclear. The compression of the Fock representation with the projection onto the direct summand $\E_e=A$ yields a \Star homomorphism from  $\Toepr(\E)^G$ onto $A$. Hence $A$ is nuclear. 

If $G$ is any group containing~$P$, the topological $G$\nb-grading of $\Toepr(\E)$ induced by the canonical normal gauge coaction of $G$ gives $\Toepr(\E)$ the structure of  reduced $\Cst$\nb-algebra of a Fell bundle over $G$. Since the fixed point algebra $\Toepr(\E)^G\cong \Toepu(\E)^G$ is the range of the conditional expectation of the gauge coaction,  (3) implies (2).
    Suppose now $G$ is amenable, in which case $\Toepu(\E) \cong \Toepr(\E)$ because their Fell bundle is amenable. If $\Toepr(\E)^G$ is nuclear, then $\Toepr(\E)$ is nuclear by \cite[Proposition 25.10]{Exel:Partial_dynamical}.  This shows that in this case \textup{(3)} is equivalent to \textup{(1)} and \textup{(2)}, completing the proof.
    \end{proof}
\end{thm}

\begin{rem}
\thmref{thm:Toep-vs-coefficient1} gives conditions under which nuclearity passes from the coefficient algebra $A$ to the Toeplitz algebra $\Toepr(\E)$. What makes this possible is that the fixed point algebra $\Toepr(\E)^G$ that appears in the intermediate step does not depend on the group in which $P$ is embedded, see \remref{rem:sameFPAforallG}. So  all it takes is to find an embedding of $P$ in an  amenable group. Examples show that such embeddings may exist even in cases when the obvious embeddings of $P$ are into nonamenable groups. 
\end{rem}

When the group~$G$ is exact,   $\Toepr(\E)$ is exact if and only if the fixed point algebra $\Toepr(\E)^G$ is exact by \cite[Proposition~25.12]{Exel:Partial_dynamical}. This  focuses our attention on  exactness of $\Toepr(\E)^G$, for which we have the following result, which is independent from exactness of $G$.

\begin{thm}\label{thm:Toep-vs-coefficient-exactness}  Resume the same assumptions of Theorem \textup{\ref{thm:Toep-vs-coefficient1}}. If $\Toepr(\E)^G$ is exact, then $A$ is exact.  The converse is true if $\E$ satisfies \eqref{eqn:independenceps1}.
\begin{proof} 
The first assertion is immediate because exactness is inherited by subalgebras. Suppose now $A$ is exact and $\E$ satisfies \eqref{eqn:independenceps1}. By \proref{pro:nuclearity-exactness-sym1},  $\mi(S)$ is exact for every $S\in \J$. The result now follows from \thmref{thm:building-subalgebras-exactness}.
  \end{proof}
\end{thm}

\section{Examples}\label{sec:examples}

In this section we exhibit several classes of monoids for which every product system of full correspondences such that the left action of the coefficient algebra contains the compact operators automatically satisfies the assumptions of~\thmref{thm:Toep-vs-coefficient1}. Most of the monoids we consider embed into amenable groups, so that \thmref{thm:Toep-vs-coefficient1} tells us that~$\Toepr(\E)$ is nuclear if and only if the coefficient algebra~$A$ is nuclear. In all our examples the monoids embed into exact groups, so that \thmref{thm:Toep-vs-coefficient-exactness} implies that $\Toepr(\E)$ is exact if and only if  $A$ is exact.
We emphasize that the results in this section apply to the class of product systems arising from the framework initially introduced by Exel as a generalization of the crossed product construction for single endomorphisms \cite{Exel:NewLook}, and developed by Larsen for abelian semigroups of endomorphisms, provided that the left action contains the compact operators
~\cite{LarETDS2010}. This is the case when the action  of $P$ on $A$ consists of injective endomorphisms with hereditary range (Example~\ref{exa:hereditary});  in fact these examples were our original motivation behind \thmref{thm:Toep-vs-coefficient1}.
\begin{example}\label{exa:hereditary} Let $A$ be a unital $\Cst$\nb-algebra and let $\alpha\colon P\to\mathrm{End}(A)$ be an action of $P$ on $A$ by injective endomorphisms with hereditary range. For each $p\in P$ let $A_{\alpha_p}\coloneqq A\alpha_p(a)$ with left action of $A$ given by multiplication, and right action implemented by $\alpha_p$. When equipped with the $A$-valued inner product $\braket{a\alpha_p(1)}{b\alpha_p(1)}\coloneqq\alpha_p^{-1}(\alpha_p(1)a^*b\alpha_p(1)),$ each $A_{\alpha_p}$ is a correspondence over~$A$ and $A_\alpha=(A_{\alpha_p})_{p\in P}$ is a product system over $P$, where the correspondence isomorphism $\mu_{p,q}$ is determined on an elementary tensor by $\mu_{p,q}(a\alpha_p(1)\otimes b\alpha_q(1))=a\alpha_p(b)\alpha_{pq}(1).$
\end{example}

\subsection{Abelian monoids} 

\begin{thm}\label{thm:abelian} Let $P$ be an abelian cancellative monoid and $A$ be a $\Cst$\nb-algebra. Let $\E=(\E_p)_{p\in P}$ be a product system over~$P$ with coefficient algebra~$A$ such that $\varphi_p(A)\supset \Comp(\E_p)$ for all $p\in P$. Suppose in addition that $\E_p$ is full for each $p\in P$.  Then $\E$ satisfies the assumptions of Theorem~\textup{\ref{thm:Toep-vs-coefficient1}},  hence $\Toepr(\E)$ is nuclear if and only if $A$ is nuclear.

\begin{proof} By assumption $P$ embeds in an abelian (hence amenable) group, and $\varphi(A)$ contains the compacts by assumption, so we only need to verify that  \eqref{eqn:I(S)includedinspan} holds for $\E$. 

First we observe that since $\E$ is full and $P$ is abelian, 
\begin{equation}\label{eq:fullness-mi}
    \mi((p,q))=\fock(\E_{p})^*\fock(\E_{q})=\fock(\E_{sp})^*\fock(\E_{sq})= \fock(\E_{ps})^*\fock(\E_{qs})=\mi((ps,qs))
    \end{equation} for all $p,q\in P$. 
   Given a not necessarily neutral word $\beta=(p_1,p_2,p_3,p_4)\in\W^2$, 
   we let $\beta'\coloneqq(p_1p_3,p_2p_3,p_3p_2,p_4p_2)\in \W^2$; then $\dot{\beta'}=\dot{\beta},$ $K(\beta')=K(\beta)$ and \eqref{eq:fullness-mi} implies that  $\mi(\beta)\subset\mi(\beta')$. Notice that the second letter coincides with the third one.

Assume as induction hypothesis that for some fixed $l\geq 3$ , and for every word $\sigma\in \W^s$ with $s\leq l$ there exists $\sigma'\in \W^s$ with $\dot{\sigma'}=\dot{\sigma}$, $K(\sigma')=K(\sigma)$,  and $\mathcal{I}(\sigma)\subset\mathcal{I}(\sigma')$ 
such that every \emph{even} letter $q_{2j}$ in $\sigma'$ coincides with the subsequent letter $q_{2j+1}$. Let $\beta=(p_1,p_2,\ldots,p_{2l+1}, p_{2l+2})\in \W^{l+1}$ be a word of even length $l+1$. Write $\beta=\sigma(p_{2l+1},p_{2l+2})$ with $\sigma=(p_1, \ldots, p_{2l-1},p_{2l})\in \W^l$ and let $\sigma'=(q_1,\ldots, q_{2l})\in \W^l$ be as in the induction hypothesis for~$\sigma$. Then $\dot{\beta}=\dot{\sigma'}p_{2l+1}^{-1}p_{2l},$ $K(\beta)=K(\dot{\sigma'}p_{2l+1}^{-1}p_{2l+2})$, and $\mi(\beta)\subset \mi(\sigma'(p_{2l+1},p_{2l+2})).$ 
Set 
\begin{align*}
\beta'&=(q_1p_{2l+1},q_2p_{2l+1},\ldots,q_{2l-1}p_{2l+1},q_{2l}p_{2l+1},p_{2l+1}q_{2l},p_{2l+2}q_{2l})\\
 &=\big(\sigma'\cdot p_{2l+1}\big) \big((p_{2l+1},p_{2l+2})\cdot q_{2l}\big).
\end{align*}
That is, $\beta'\in \W^{l+1}$ is the concatenation of the word obtained by multiplying  each letter of $\sigma'$ by $p_{2l+1}$ on the right, with the word $(p_{2l+1}q_{2l},p_{2l+2}q_{2l})$. Then $\beta'$ satisfies $\dot{\beta'}=\dot{\beta}$, $K(\beta')=K(\sigma'(p_{2l+1},p_{2l+2}))=K(\beta)$,  and again by \eqref{eq:fullness-mi}, 
\[
\mi(\beta)\subset \mi(\sigma'(p_{2l+1},p_{2l}))\subset\mi(\beta').
\]
Moreover, since $P$ is abelian, $q_{2l}p_{2l+1} = p_{2l+1}q_{2l}$ and thus every even letter of $\beta'$  coincides with the subsequent letter. This proves that the induction hypothesis holds for words of even length $l+1$, and hence by induction it also holds for words of arbitrary length. 

Now let $\beta$ be a neutral word. Then there exists 
a neutral word $\beta'$, with the same even length as $\beta$, such that $K(\beta')=K(\beta),$ $\mi(\beta)\subset \mi(\beta')$ and every even letter in $\beta'$ coincides with the subsequent letter. Because $\beta'$ is neutral, it follows in addition that the first and last letters of $\beta'$ are the same. Hence we may view $\beta'$ as a word in $\langle\W_{\rm{sym}}\rangle$ by adding letters that equal to the identity element~$e$, if necessary. This shows that $\E$ satisfies \eqref{eqn:I(S)includedinspan}, and thus all the  assumptions of \thmref{thm:Toep-vs-coefficient1} are satisfied. By the equivalence of (1) and (3) in \thmref{thm:Toep-vs-coefficient1}, $\Toepr(\E)$ is nuclear if and only if $A$ is nuclear, completing the proof.
    \end{proof}
    
\end{thm}

\subsection{ax+b monoids}
Suppose now $R$ is an integral domain and let $R^\times$ denote the multiplicative monoid of nonzero elements of $R$. Let  $R\rtimes R^\times$ be the  $ax+b$-monoid of  $R$, in which the operation is $(b,a) (d,c) = (b+ad,ac) $.

\begin{thm} Let  $R\rtimes R^\times$ be the  $ax+b$-monoid of  the integral domain $R$, and let $A$ be a $\Cst$\nb-algebra. Let $\E=(\E_p)_{p\in P}$ be a product system over~$R\rtimes R^\times$ with coefficient algebra~$A$ such that $\varphi_p(A)\supset \Comp(\E_p)$ for all $p\in R\rtimes R^\times$. Suppose in addition that $\E_p$ is full for each $p\in R\rtimes R^\times$. Then $\E$ satisfies the assumptions of \thmref{thm:Toep-vs-coefficient1}, so that  $\Toepr(\E)$ is nuclear if and only if $A$ is nuclear.  
\begin{proof} 
Let $K = (R^\times)\inv R$ be the fraction field of $R$. Then $R\rtimes R^\times$ embeds in the group $K\rtimes K^*$, which is amenable. Left multiplication by $A$ contains the compacts by assumption, so we only need to verify that the inclusion \eqref{eqn:I(S)includedinspan} holds.

Notice that since $\{(a,1)\mid a\in R\}$ is an abelian submonoid of $R\rtimes R^\times$, it follows as in the proof of \thmref{thm:abelian} that 
\[
\mi(\beta)\subset\clsp\{b\mid b\in \mi(\beta'), \ \beta'\in W_{\rm{sym}}(R\rtimes\{1\}), K(\beta')=K(\beta)\}
\]
in case $b_i=1$ for all $i=1,\ldots, 2l$. Also, in case $b_i=1$ for all~$i$ except for $i=1$ and $i=2l$, we have $b_1=b_{2l}$ and 
\[
\mathcal{I}(\beta)\subset\tilde{\fock}(\E_{(0,b_1)})^*\mathcal{I}(\tilde{\beta})\tilde{\fock}(\E_{(0,b_1)}),
\]
where $\tilde{\beta}=((a_1,1), (a_2,1),\ldots, (a_{2l-1},1),(a_{2l},1))$. Since $\tilde{\beta}$ is neutral because~$\beta$ is, we obtain that
\[
\mi(\beta)\subset\clsp\{b\mid b\in \mi(\beta'),\  \beta'\in W_{\rm{sym}}(R\rtimes R^\times), K(\beta')=K(\beta)\}.
\]

The general case follows from the above if we  show that for every word $\beta\in \W$ (not necessarily neutral) there exist $\tilde{\beta}\in \W^l(R\rtimes\{1\})$ and $c_1, c_{2l}\in R^\times$ such that $\dot{\beta}=(0,c_1)^{-1}\dot{\tilde{\beta}}(0,c_{2l})$, $K(\beta)\subset K(((0,c_1),e)\tilde{\beta}(e,(0,c_{2l})))$ and $\mathcal{I}(\beta)\subset \tilde{\fock}(\E_{(0,c_1)})^*\mathcal{I}(\tilde{\beta})\tilde{\fock}(\E_{(0,b_1)})Q_{K(\beta)}. $ To do so we begin by proving by induction that given a word $\beta=(p_1,p_2,\ldots, p_{2l-1}, p_{2l})\in \W^l$ in  $R\rtimes R^\times$, with  $p_i=(a_i,b_i)\in R\rtimes R^\times$ for each $i=1, \ldots, 2l$, there exists a word $\beta'=(q_1,q_2,\ldots, q_{2l-1},q_{2l})\in \W^l,$ where $q_i=(c_i,d_i)$ for $i=1,\ldots, 2l$, with $\dot{\beta'}=\dot{\beta}$,  $K(\beta')=K(\beta)$ and $\mi(\beta)\subset\mi(\beta')$ such that for every even letter $q_{2j}$ in $\beta'$, its coordinate $d_{2j}$ coincides with the corresponding coordinate $d_{2j+1}$ in the subsequent letter $q_{2j+1}$.

As in the proof of \thmref{thm:abelian}, we have that for all $a_1, a_2\in R$ and $b_1,b_2\in R^\times$ 
    $$\mi(((a_1,b_1),(a_2,b_2)))=\mi(((da_1,db_1),(da_2,db_2)))$$
for all $d\in R^\times$ since $\E_{(0,d)}$ is full. So given a word $\beta=(p_1,p_2,p_3,p_4)\in\W^2$ with $p_i=(a_i,b_i)\in R\rtimes R^\times$, 
   we let $\beta'\coloneqq((0,b_3)p_1,(0,b_3)p_2,(0,b_2)p_3,(0,b_2)p_4)\in \W^2$; then $\dot{\beta'}=\dot{\beta},$ $K(\beta')=K(\beta)$ and $\mi(\beta)\subset\mi(\beta')$. In addition, if we write $\beta'=(q_1,q_2,q_3,q_4)$ and $q_i=(c_i,d_i)\in R\rtimes R^\times $, we see that  $d_2=d_3$ because $R$ is commutative, so that $\beta'$ has the desired properties.

Assume as induction hypothesis that for some fixed $l\geq 3$, for every word $\sigma\in \W^s$ with $s\leq l$ there exists $\sigma'\in \W^s$ with $\dot{\sigma'}=\dot{\sigma}$, $K(\sigma')=K(\sigma)$,  and $\mathcal{I}(\sigma)\subset\mathcal{I}(\sigma')$ such that the second coordinate of every even letter $q_{2j}$ in $\sigma'$ coincides with the second coordinate in the subsequent letter $q_{2j+1}$. Let $\beta=(p_1,p_2,\ldots,p_{2l+1}, p_{2l+2})\in \W^{l+1}$ be a word of even length $l+1$. Proceeding as in the proof of \thmref{thm:abelian}, we write $\beta=\sigma(p_{2l+1},p_{2l+2})$ with $\sigma=(p_1, \ldots, p_{2l-1},p_{2l})\in \W^l$ and let $\sigma'=(q_1,\ldots, q_{2l})\in \W^l$ with $q_i=(c_i,d_i)\in R\rtimes R^\times$ be as in the induction hypothesis for~$\sigma$. Putting \begin{align*}
\beta'&=\big((0,b_{2l+1})\cdot\sigma'\big) \big((0,d_{2l})\cdot(p_{2l+1},p_{2l+2}) \big)
\end{align*} we see that $\beta'$ satisfies the required properties, so that the induction hypothesis also holds for $\beta$. Hence it also holds for every word in $\W(R\rtimes R^\times)$. 

Now let $\beta=(p_1,p_2,\ldots, p_{2l-1}, p_{2l}) \in \W^l$. We write $p_i=(a_i,b_i)\in R\rtimes R^\times$ for $i=1, \ldots, 2l$. By the above we may assume that $b_{2i}=b_{2i+1}$ for $1\leq i\leq l-1.$ Using that $\fock(\E_{(a,b)})=\fock(\E_{(a,1)})\fock(\E_{(0,b)})$ on the Fock space and that for all $b\in R^\times$ the left action of~$A$ on $\E_{(0,b)}$ contains the compact operators, we deduce that $\mi(\beta)\subset \mi(\beta')Q(\beta),$ where $\beta'=(q_1,q_2,\ldots, q_{2l})$ is the word whose first and last letters coincide with those of $\beta$, and $q_i=(a_i,1)$ for $i=2,3,\ldots, 2l-1$.  Since $\beta'$, in turn, satisfies $\mi(\beta')=\mi\big(((0,b_1),e)\tilde{\beta}(e,(0,b_{2l}))\big)$ where $\tilde{\beta}=((a_1,1),(a_2,1),\ldots, (a_{2l},1))\in \W(R\rtimes\{1\}),$ the conclusion now follows as in \thmref{thm:abelian}.
\end{proof}

\end{thm}

\subsection{Baumslag--Solitar monoids}

For each pair of nonzero integers $m,n$ the Baumslag-Solitar group $BS(m,n)$ is
\[BS(m,n) \coloneqq \langle a,b\mid ab^ma\inv = b^n \rangle .\]
 The presentation can be restated without using negative exponents, so that it makes sense for monoids. Thus, when  $m,n \geq 1$ the Baumslag--Solitar monoid is defined 
by \[BS^+(m,n) \coloneqq \langle a,b\mid ab^m = b^n a\rangle_+,\]
and when $m\geq 1$ and $n\leq -1$ it is defined by
\[BS^+(m,n) \coloneqq \langle a,b\mid b^{-n}ab^m = a\rangle_+.\]
The cases  $m\leq -1$, $n\leq -1$ and $m\leq -1$, $n\geq 1$ can be reduced to the two above by exchanging $m$ and $n$.
 We will thus assume that we are either in  
 
 Case (1): $m \geq 1$, $n \geq 1$ or in 
 
 Case (2): $m\geq 1$, $n\leq -1$.

Every element~$p$ in $BS^+(m,n)$ has a unique normal form 
\[
p=b^{c_0}ab^{c_1}a\ldots b^{c_{k-1}}a b^{c_k},
\]
in which  $k \in \NN$ and  $0\leq c_i\leq |n|-1 $ for $0\leq i\leq k-1$ with $c_k\in \NN$ in case (1) and $c_k\in \ZZ$ in case (2), except that $c_0 \geq 0$ always. This is obtained by pushing powers of $b$ to the right using $ b^na \rightsquigarrow ab^m $ until the normal form is achieved, see e.g. \cite[Section~2]{Spielberg12}. Here $k$ is the number of $a$'s in any  expression of $p$; it is called the {\em height} of $p$ and is denoted $\theta(p)\coloneqq k\in \NN$. 
The {\em stem} of $p$ is the product $\mathrm{stem}(p)\coloneqq b^{c_0}ab^{c_1}a\ldots b^{c_{k-1}}a$, so that $p=\mathrm{stem}(p)b^{c_k}$.

Using the normal form for elements of $BS(m,n)$ one can show that the monoid $BS^+(m,n)$ embeds canonically in the group $BS(m,n)$ as the submonoid generated by $a$ and $b$.
The pair $(BS(m,n),BS^+(m,n))$ is a weak quasi-lattice order, that is, $BS^+(m,n)$ is a group-embeddable right LCM monoid with no nontrivial invertible elements. 

\begin{rem}\label{rem:upper-bounds} 
Recall from  \cite[Lemma 5.5]{aHRT} that if $p$ and $q$ have a common upper bound in $BS^+(m,n)$ (i.e. if $pBS^+(m,n)\cap qBS^+(m,n)\neq\emptyset$) and $\theta(p)\leq \theta(q)$, then $\rm{stem}(p)$ is a prefix of $\rm{stem}(q)$. This will be useful in the proof of the next result.
\end{rem} 

We are now ready to show that every full product system  over a Baumslag--Solitar monoid with $m,n\geq 1$ such that the left action contains the compact operators automatically satisfies the assumptions in \thmref{thm:Toep-vs-coefficient1}. For the case $m\geq 1$, $n\leq -1$ see \remref{rem:case-mn-negative}.

\begin{thm}\label{thm:Baumslag} Suppose $m,n\geq 1$ and let  $\E=(\E_p)_{p\in BS^+(m,n)}$ be a product system over the
Baumslag--Solitar monoid~$BS^+(m,n)$ with coefficient algebra~$A$ such that $\varphi_p(A)\supset \Comp(\E_p)$ for all $p\in BS^+(m,n)$. Suppose in addition that $\E_p$ is full for each $p\in BS^+(m,n)$. Then~$\E$ satisfies the assumptions of \thmref{thm:Toep-vs-coefficient1}, and the following hold:
\begin{enumerate}
    \item $\Toepr(\E)$ is exact if and only if $A$ is exact;
    \smallskip\item if $\gcd(m,n) =1$, then $\Toepr(\E)$ is nuclear if and only if $A$ is nuclear.
    \end{enumerate}
    \begin{proof} Let $\beta=(p_1,p_2,\ldots, p_{2l})\in \W^l$ be a neutral word in $BS^+(m,n).$ We will prove by induction on $l$ that there exists $\beta'\in \langle\W_{\rm{sym}}\rangle$, with even length at most $l$, $K(\beta)\subset K(\beta')$ and $\mi(\beta)\subset \mi(\beta')Q_{K(\beta)}.$ 

    First notice that because $\fock(\E_p)^*\fock(E_p)\subset \fock(\braket{\E_p}{\E_p})$, using Remark \ref{rem:upper-bounds} we may assume that for each $1\leq i\leq l$ one of the letters $p_{2i-1}$ or $p_{2i}$ has height zero, and hence belongs to $\langle b\rangle ^+$. By the same reasoning we may assume that the letter of positive height starts with~$a$ in normal form.
    
       Now let $\beta$ as above with $l=2$, say $\beta=(p_1,p_2,p_3,p_4)$. Since $\beta$ is neutral, we must have either $\theta(p_2)=\theta(p_3)=0$ or $\theta(p_1)=\theta(p_4)=0.$

       \textbf{Case 1:} $\theta(p_2)=\theta(p_3)=0$. In this case $\beta=(p_1, b^{k_1}, b^{k_2}, p_4)$ with $k_1, k_2\in \NN$. If $k_2\geq k_1$, we use that the action of $A$ on $\E_{b^{k_1}}$ contains the compact operators to conclude that $$\mi(\beta)\subset \mi(\beta')Q(\beta),$$ where $\beta'=(b^{k_2-k_1}p_1, p_4)$ is neutral with $K(\beta)\subset K(\beta')$. Similarly in case $k_1\geq k_2$ we get $\mi(\beta)\subset \mi(\beta')Q(\beta)$, with $\beta'=(p_1, b^{k_1-k_2}p_4)$.

          \textbf{Case 2:} $\theta(p_1)=\theta(p_4)=0$. In this case $\beta=(b^{k_1}, p_2, p_3, b^{k_2})$ with $k_1, k_2\in \NN$. Since $\beta$ is neutral, we obtain $p_2p_3\inv=b^{k_1-k_2}.$ In case $k_1\geq k_2$, we use that $p_2=b^{k_1-k_2}p_3$ and that the left action of $A$ on $\E_{p_3}$ contains the compact operators to get that $\mi(\beta)\subset \mi(\beta')Q_{K(\beta)}$ with $\beta'=(b^{k_1}, b^{k_1})$. Similarly, in case $k_1-k_2<0$ we see that $p_3=b^{k_2-k_1}p_2$, and so $\mi(\beta)\subset \mi(\beta')Q_{K(\beta)}$ with $\beta'=(k_2,k_2).$

Suppose $\beta=(p_1,p_2,\ldots, p_{2l+1}, p_{2l})\in\W$ is a neutral word in $BS^+(m,n)$ as above with even length $l+1\geq 3$. We separate the proof into two cases.

\textbf{Case 1:} $\theta(p_1)>0$. In this case there must be $1\leq i\leq l$ such that $\theta(p_{2i})=\theta(p_{2i+1})=0$ because $\beta$ is neutral. Then we proceed as in Case 1 above to find a neutral word $\beta_l\in \W^l$ such that $K(\beta)\subset K(\beta_l)$ and $\mi(\beta)=\mi(\beta_l)Q_{K(\beta)}.$ By induction there exists $\beta'\in \langle\W_{\rm{sym}}\rangle$ with $K(\beta)\subset K(\beta_l)\subset K(\beta') $ and $\mi(\beta)\subset \mi(\beta_l)Q_{K(\beta)}\subset \mi(\beta')Q_{K(\beta)}$ as wanted.

\textbf{Case 2:} $\theta(p_1)=0$.  In this case there must be $1\leq i\leq l$ such that $\theta(p_{2i})\neq 0$ and $\theta(p_{2i+1})\neq0$. Write $\beta$ as the concatenation $\beta=\sigma_1\beta_1\sigma_2$ where $\beta_1=(p_{2i-1},p_{2i}, p_{2i+1},p_{2i+2})$, $\sigma_1\in \W^{i-1}$, and $\sigma_2\in \W^{l-i}.$ It follows from our assumptions on $\beta$ that $\theta(p_{2i-1})=0=\theta(p_{2i+2}).$ If either $\theta(p_{2i-2})=0$ or $\theta(p_{2i+3})=0$ we proceed as in the proof of Case 1.

Suppose $\theta(p_{2i-2})\neq 0$. Let $k\geq 0$ be such that $p_{2i-1}=b^k$. Using the relations in $BS^+(m,n)$ we can find a positive power $c\in \NN$ such that $b^cp_{2i-3}=p_{2i-3}'b^d$ for some $p'_{2i-3}\in BS^+(m,n)$ and some $d\geq k$. Then we use that $$\mi((p_{2i-3},p_{2i-2}))\subset \mi((b^cp_{2i-3},b^cp_{2i-2})),$$ that $\fock(\E_{p_{2i-3}'b^{d}})=\fock(\E_{p_{2i-3}'b^{d-k}})\fock(\E_{b^k})$ and that the action of $A$ on $\E_{b^k}$ contains the compact operators to get a neutral word $\beta_l$ of even length~$l$ such that $K(\beta)\subset K(\beta_l)$ and $\mi(\beta)\subset \mi(\beta_l)Q_{K(\beta)}.$ Again by induction there exists $\beta'\in \langle\W_{\rm{sym}}\rangle$ with $K(\beta)\subset K(\beta_l)\subset K(\beta') $ and $\mi(\beta)\subset \mi(\beta_l)Q_{K(\beta)}\subset \mi(\beta')Q_{K(\beta)}$ as required. This gives the case when the last letter of $\sigma_1$ has positive height. The case when the first letter of $\sigma_2$ has positive height follows from this by taking adjoints.

      Part (1) now follows from 
\thmref{thm:Toep-vs-coefficient-exactness} and \cite[Proposition~25.12]{Exel:Partial_dynamical} because $BS(m,n)$ is a one-relator group hence is  exact by \cite[Theorem~2.1]{G2002}. Part (2) follows from \thmref{thm:Toep-vs-coefficient1} because when $\gcd(m,n) =1$ there is an embedding of $B^+(m,n)$ into an amenable group $G$ by \proref{pro:bs-amenable-emb}.
    \end{proof}
\end{thm}

\begin{rem}\label{rem:case-mn-negative}
In case $m\geq 1$, $n\leq -1$, every two elements $p,q\in B^+(m,n)$ with a common upper bound in $BS^+(m,n)$ are comparable. See, for example, \cite[Proposition~5.11]{ABCD}. Hence every product system over~$B^+(m,n)$ is automatically compactly aligned, and so \corref{cor:compactly-align} applies.
    \end{rem}

\subsection*{The monoid $BS^+(m,n)$ embeds in an amenable group if $\gcd(m,n) =1$}
\begin{prop}\label{pro:bs-amenable-emb}
For every   $m,n \geq 1$ there is a unique monoid homomorphism 
\[
\pi: BS^+(m,n) \to M_n^+(\ZZ) \subset \operatorname{GL}_2(\QQ)
\]
such that 
\[
a \mapsto A := \matr{n}{0}{m}\qquad \text {and} \qquad b\mapsto 
B:= \matr{1}{1}{1}.
\]
 If  $\gcd(m,n) =1$, then $\pi$ is an embedding of $BS^+(m,n)$ into the subgroup of $\operatorname{GL}_2(\QQ)$ generated by $A$ and $B$, which is amenable.

\begin{proof}
For $k,l \in \NN$, we compute    $AB^k = \matr{n}{kn}{m}$
and
$B^l A = \matr{n}{lm}{m}$.

Let $d:=\gcd(m,n)$, and set  $k =  m/d$ and $l = n/d$; then 
\[
AB^{m/d} = \matr{n}{mn/d}{m}
 = B^{n/d}A.
\]

If $m,n$ are relatively prime, this is exactly the relation defining  $BS^+(m,n)$. In general write $B^m = (B^{m/d})^d$ and move the $d$ blocks $B^{m/d}$ one by one using the above relation to obtain the relation defining $B^+(m,n)$.

Suppose now that the word  $p(a,b)$ in $a$ and $b$  representing an element of $B^+(m,n)$ is  in (R) normal form. The (R) normal form is the one obtained by pushing every possible $b$ to the right until the substitution $ b^na \rightsquigarrow ab^m $ can no longer be applied.
Specifically, 
\[
p(a,b) =  b^{j_0} ab^{j_1} a b^{j_2} a \ldots ab^{j_\ell} 
\]
with  $0\leq j_i < n$ for each $i = 0,1, 2, \ldots , \ell-1$ and $j_\ell\in \ZZ$ (in fact, $j_\ell\in\NN$ in case (1)), where $\ell$ is the number of $a$'s in $p(a,b)$. 
The corresponding product of matrices 
is 
\[
p(A,B) = \matr{n^\ell}{ j_0 m^{\ell} + n j_1 m^{\ell-1} + \cdots +n^{\ell -1}j_{\ell -1} m + n^{\ell} j_\ell }{m^\ell}
\]

Assume now that there exist two different normal forms $p$ and $p'$ (necessarily corresponding to two different elements of $BS^+(m,n)$  that produce the same matrix $p(A,B) = p'(A,B)$. 
Then $p(A,B)_{1,1} = p'(A,B)_{1,1}$  implies that $\ell = \ell'$.
If we then take  $p(A,B)_{1,2} = p'(A,B)_{1,2}$ modulo $n$ we get
\[
(j_0 - j'_0) m^\ell   = 0 \pmod n.
\]

Assuming $\gcd(m,n) =1$ we see that $j_0 - j'_0 =0 \pmod n$, and hence $j_0 =j'_0$ because both $j_0$ and $j'_0$ are between $0$ and $n-1$. We can then factor out the first power of $B$ and the first $A$ in both $p$ and $p'$.
This gives two different shorter normal forms in $BS^+(m,n)$ (in the $a$-length) that produce the same matrix. Iterating this procedure, we arrive at a contradiction.
We conclude that the representation is an embedding when $d=1$.

    The subgroup of $\operatorname{GL}_2(\QQ)$ generated by the matrices $A$ and $B$ is contained in the 
   group of upper triangular matrices, which is solvable, hence  amenable.
\end{proof}
\end{prop}

\begin{rem}
    If $\gcd(m,n) \neq 1$, then the representation $\pi$ given in the proposition is not an embedding of $BS^+(m,n)$ into matrices.  Indeed, as seen above, the matrices satisfy $AB^{m/d} =  B^{n/d}A$, but since $n/d <n$
both $ab^{m/d}$ and $b^{n/d}a$ are in normal form, hence are different in $BS^+(m,n)$.
If we factor $A$ as  
\[
A := \matr{d}{0}{d} \matr{n/d}{0}{m/d}
\]
and use a slight modification of the proof of the lemma we see that the image of $\pi$ is isomorphic to $BS^+(m/d,n/d)$.
 The case $m=n= d$ presents an extreme failure of embeddability; indeed  
\[
 A := \matr{d}{0}{d}\qquad \text {and} \qquad 
B:= \matr{1}{1}{1}
\]
commute and generate a copy of $\NN^2$. We believe that a better understanding of this case could lead to a general amenable embedding of $BS^*(m,n)$.
\end{rem}

\end{document}